\documentclass{amsart}
\usepackage[utf8]{inputenc}
\usepackage{epsdice}
\usepackage{enumitem}
\usepackage{color}

\usepackage{relsize}
\usepackage[bbgreekl]{mathbbol}
\usepackage{amsfonts}
\DeclareSymbolFontAlphabet{\mathbb}{AMSb} %to ensure that the meaning of \mathbb does not change
\DeclareSymbolFontAlphabet{\mathbbl}{bbold}
\newcommand{\Prism}{{\mathlarger{\mathbbl{\Delta}}}} 

\usepackage[OT2,T1]{fontenc}
\usepackage[final]{microtype}
\usepackage{tikz}
\usepackage{tikz-cd}
\usepackage{xypic}

\usepackage[draft=false,linktocpage=true]{hyperref}
\usepackage[capitalise]{cleveref}

\usepackage{mathtools}
\usepackage{amsmath}

\usepackage{amsthm,amssymb}
\numberwithin{equation}{section}

\theoremstyle{plain}
\newtheorem{theorem}[equation]{Theorem}

\newtheorem{proposition}[equation]{Proposition}
\newtheorem{lemma}[equation]{Lemma}

\newtheorem{corollary}[equation]{Corollary}
\newtheorem{porism}[equation]{Porism}

\theoremstyle{definition}
\newtheorem{definition}[equation]{Definition}
\newtheorem{notation}[equation]{Notation}

\newtheorem{example}[equation]{Example}

\newtheorem{remark}[equation]{Remark}

\title{Integral \(p\)-adic Hodge filtrations in low dimension and ramification}

\author{Shizhang Li}

\address{Department of Mathematics,University of Michigan, 530 Church Street,
  Ann Arbor, MI 48109}

\email{shizhang@umich.edu}

\begin{document}

\maketitle

\begin{abstract}
Given an integral \(p\)-adic variety, we observe that if the integral
Hodge--de Rham spectral sequence behaves nicely, then the special fiber
knows the Hodge numbers of the generic fiber.
Applying recent advancements of integral \(p\)-adic Hodge theory,
we show that such a nice behavior is guaranteed
if the \(p\)-adic variety can be lifted to an analogue of second Witt vectors
and satisfies some bound on dimension and ramification index.
This is a (ramified) mixed characteristic 
analogue of results due to Deligne--Illusie, Fontaine--Messing, and Kato.
Lastly, we discuss an example illustrating the necessity of
the aforementioned lifting condition, which is of independent interest.
\end{abstract}

\tableofcontents

% * Introduction
\section{Introduction}
Given a smooth proper scheme \(\mathcal{X}\) over some \(p\)-adic ring of integers \(\mathcal{O}_K\),
can we tell the Hodge diamond of its generic fiber \(X\) by simply staring at the geometry of the special fiber \(\mathcal{X}_0\)?
In general, there is no hope of this being true.
But surely if one puts some constraints, this will be true.

There are two typical pathological phenomenons concerning Hodge and de Rham cohomology groups in an integral \(p\)-adic situation:
one being torsions in the cohomology groups, the other being the
non-degeneracy of the integral Hodge--de Rham spectral sequence.
In this paper we convey the idea that, 
for the question asked in the beginning, the trouble comes from the second phenomenon.
To be more precise, we define virtual Hodge numbers for
any smooth proper variety
in characteristic \(p\), see~\cref{virtual Hodge}.
In~\cref{equality of numbers}, we observe that 
if the integral Hodge--de Rham spectral sequence
of a lift \(\mathcal{X}\) degenerates saturatedly 
(see~\cref{definition of SS degeneration}),
then the virtual Hodge numbers of \(\mathcal{X}_0\) agree with the Hodge numbers of
the generic fiber \(X\).

Then we show the following:
\begin{theorem}[Main Theorem]
\label{main theorem}
Let \(\mathcal{X} \to \mathrm{Spf}(\mathcal{O}_K)\) be a smooth proper formal scheme,
let $\mathrm{W}(\kappa)$ be the Witt ring of the residue field of $\mathcal{O}_K$.
Let \(\mathfrak{S} \coloneqq \mathrm{W}(\kappa)[\![u]\!]\) 
be the Breuil--Kisin prism associated with \(\mathcal{O}_K\)
and let \(E\) be an associated Eisenstein polynomial, see \cite[Example 1.3.(3)]{BS19}.
Assume that
\begin{enumerate}
    \item there is a lift of \(\mathcal{X}\) over \(\mathfrak{S}/(E^2)\); and
    \item the relative dimension of \(\mathcal{X}\) 
    and the ramification index \(e\) of \(\mathcal{O}_K\) satisfy the inequality: \((\dim(\mathcal{X}_0) + 1) \cdot e < p-1\).
\end{enumerate}
Then the Hodge--de Rham spectral sequence for \(\mathcal{X}\) is split degenerate.
In particular, we have equality of (virtual) Hodge numbers:
\[
\mathfrak{h}^{i,j}(\mathcal{X}_0) = h^{i,j}(X).
\]
\end{theorem}
One may think of this result as a mixed characteristic analogue of 
a theorem by Deligne--Illusie~\cite{DI87}.
There is a similar statement with weaker conclusion when \(\dim X\)
exceeds the bound in (2), see~\cref{porism}.
In the case when \(\mathcal{O}_K\) is the Witt ring of a perfect field, condition (1)
is automatic and our result can also be deduced
from Fontaine--Messing's result~\cite{FM87}
or Kato's result~\cite{K87}.
We summarize the implication relations between various
relevant conditions in~\cref{summary}.
The proof of the main theorem uses recent theory of prismatic cohomology
due to Bhatt--Scholze~\cite{BS19} along with a result of Min~\cite{Min19}.
For more details, see~\cref{Consequences}.

Lastly, one may wonder if either condition (1) or (2) is really necessary.
Previously in~\cite{Li18} we have constructed a pair of relative \(3\)-folds over 
\(\mathbb{Z}_p[\zeta_p]\)
with isomorphic special fiber,
such that their generic fibers have different Hodge numbers,
showing that the conclusion of \cref{main theorem} is not true in general.
While it is unclear if condition (2) is really necessary,
in~\cref{lifting ABM}, which is independent from other sections,
we discuss extensively an example illustrating the necessity of condition (1).
\begin{theorem}[{see~\cref{main example}}]
There exists a smooth projective relative \(4\)-fold \(\mathcal{X}\)
over a ramified degree two extension \(\mathcal{O}_K\) of \(\mathbb{Z}_p\), such that
both of its Hodge--de Rham and Hodge--Tate spectral sequences are non-degenerate.
Moreover the Hodge/conjugate filtrations are non-split as \(\mathcal{O}_K\)-modules.
\end{theorem}
The idea of construction, which may be traced back to W.~Lang's work~\cite{Lang95}
and Raynaud's~\cite{Ray79}, is as follows.
The exotic group scheme \(\alpha_p\) admits liftings over ramified \(p\)-adic rings
of integers.
We choose a lift \(G\) over a degree \(2\) ramified \(p\)-adic ring of integers,
then we study the Hodge--de Rham and Hodge--Tate spectral sequences
of \(BG\), the classifying stack of \(G\).
With the aid of various computations in~\cite{ABM19},
we find out that both are non-degenerate starting at degree \(3\)
with non-split Hodge/conjugate filtrations starting at degree \(2\).
In the end, we take an approximation of \(BG\) to get the desired example.

\subsection*{Acknowledgement}
The author cannot overstate how much this paper owes to Bhargav Bhatt for so many helpful discussions.
He is also grateful to Yu Min for useful conversations.
The author would also like to thank
Johan de Jong, Remy van Dobben de Bruyn, Arthur-C\'{e}sar Le Bras, Tong Liu, 
Qixiao Ma, Emanuel Reinecke, and Jason Starr,
for their interests in and correspondence about this project.
The author thanks anonymous referees, for reminding the author of using duality statements,
pointing out many references as well as several precious comments.

% * Preliminary
\section{Preliminaries on spectral sequences over DVRs}
This section is a general discussion of spectral sequences associated with a filtered bounded 
perfect complex over a DVR.

\begin{notation}
Throughout this section, let \(R\) be a DVR with a uniformizer \(\pi\).
Denote \(K \coloneqq R[1/\pi]\) and \(\kappa \coloneqq R/\pi\).
Let \((C,\mathrm{Fil}^{\bullet})\) be a filtered object in \(D^b_{\mathrm{Coh}}(R)\).
We assume the filtration on \(C\) to be exhaustive and complete.

Given a finitely generated \(R\)-module \(M\), we denote by \(M_{\mathrm{tor}}\) the torsion submodule in \(M\),
and we denote its torsion-free quotient by \(M_{\mathrm{tf}} \coloneqq M/M_{\mathrm{tor}}\).
\end{notation}

\begin{remark}
We do not assume this filtration to be either increasing or decreasing, as it is modeling both 
Hodge--de Rham and Hodge--Tate spectral sequences.
\end{remark}

From \((C,\mathrm{Fil}^{\bullet})\) we naturally get a spectral sequence converging from 
\(H^i(\mathrm{Gr}^j)\) to \(H^i(C)\).
From now on, we will call it ``the spectral sequence" if no confusion seems to arise.
In the following definition, we refine the classical notion of the spectral sequence being degenerate.

\begin{definition}
\label{definition of SS degeneration}
\leavevmode
\begin{enumerate}
    \item We say the spectral sequence \emph{degenerates} or \emph{is degenerate} if for all pairs of integers \((i,j)\),
    the natural map \(\mathrm{H}^i(\mathrm{Fil}^j) \to \mathrm{H}^i(C)\) is an injection;
    \item We say the spectral sequence \emph{degenerates saturatedly} or \emph{is saturated degenerate} if it degenerates and
    the induced injection \(\mathrm{H}^i(\mathrm{Fil}^j)_{\mathrm{tf}} \to \mathrm{H}^i(C)_{\mathrm{tf}}\) is saturated; and
    \item We say the spectral sequence \emph{degenerates splittingly} or \emph{is split degenerate} if it degenerates and
    the induced injection \(\mathrm{H}^i(\mathrm{Fil}^j) \to \mathrm{H}^i(C)\) splits.
\end{enumerate}
\end{definition}

Recall that an injection/inclusion of torsion-free \(R\)-modules \(N \subset M\) is said to be \emph{saturated} if we have \(\pi N = N \cap \pi M\)
or, what is the same, the quotient \(M/N\) is \(\pi\)-torsion-free.

\begin{remark}
It is obvious that the spectral sequence being split degenerate implies it being saturated degenerate, and both implies it is degenerate.
\end{remark}

In the case of Hodge--de Rham or Hodge--Tate spectral sequences (over mixed characteristic DVRs), 
we know that they degenerate after inverting \(p\) (see~\cite[Corollary 1.8]{Sch13} and~\cite[Theorem 1.7]{BMS1}).
In this scenario, we have a condition on the infinite-page of the spectral sequence characterizing 
the spectral sequence being saturated or split degenerate:

\begin{proposition}
\label{characterizing degenerations}
Suppose that the spectral sequence degenerates after inverting \(\pi\).
Then 
\begin{enumerate}
    \item the spectral sequence is saturated degenerate if and only if
    \[
    \mathrm{length}(H^i(C)_{\mathrm{tor}}) = \sum_{j} \mathrm{length}(H^i(\mathrm{Gr}^j C)_{\mathrm{tor}})
    \]
    for all \(i\); and
    \item the spectral sequence is split degenerate if and only if
    there is an abstract isomorphism of \(R\)-modules:
    \[
    H^i(C)_{\mathrm{tor}} \simeq \bigoplus_{j} H^i(\mathrm{Gr}^j C)_{\mathrm{tor}}
    \]
    for all \(i\).
\end{enumerate}
\end{proposition}

Note that we assumed the \(\mathrm{Fil}^{\bullet}\) to be exhaustive and saturated,
the summation process is finite.

\begin{proof}[Proof of~\cref{characterizing degenerations}(1)]
First notice that 
\[
\mathrm{length}(H^i(C)_{\mathrm{tor}}) \leq \sum_{j} \mathrm{length}(\mathrm{Gr}^j(H^i(C))_{\mathrm{tor}}) 
\leq \sum_{j} \mathrm{length}(H^i(\mathrm{Gr}^j C)_{\mathrm{tor}}),
\]
where the second inequality comes from the fact that the spectral sequence degenerates after inverting \(\pi\)
(so \(\mathrm{Gr}^j(H^i(C))_{\mathrm{tor}}\) must be a subquotient of \(H^i(\mathrm{Gr}^j C)_{\mathrm{tor}}\)).
Therefore the equality condition implies equality between \(\mathrm{Gr}^j(H^i(C))_{\mathrm{tor}}\) and
\(H^i(\mathrm{Gr}^j C)_{\mathrm{tor}}\). 
In other words, every element in \(H^i(\mathrm{Gr}^j C)_{\mathrm{tor}}\) is a permanent cycle.
Since the spectral sequence degenerates after inverting \(\pi\), we know all the differentials in the spectral sequence
are torsion. 
Combining these two, we see that all the differentials are forced to be zero, which exactly means that the spectral sequence must degenerate.

Now we have reduced the statement of (1) to: assume the spectral sequence degenerates,
then it is saturated degenerate if and only if the equality of lengths of \(\pi\)-torsions hold, 
and this statement is handled in the following~\cref{filtered module lemma}.

\begin{lemma}
\label{filtered module lemma}
Let \(M\) be a finitely generated \(R\)-module with an exhaustive and saturated filtration \(\mathrm{F}^{\bullet}\).
Then \(\mathrm{F}^i_{\mathrm{tf}} \subset M_{\mathrm{tf}}\) is saturated for all \(i\) if and only if
\[
\mathrm{length}(M_{\mathrm{tor}}) = \sum_{i} \mathrm{length}(\mathrm{Gr}^i)_{\mathrm{tor}}).
\]
\end{lemma}

\begin{proof}[Proof of the~\cref{filtered module lemma}]
First observe that for any \(i\) we have an inequality
\[
\mathrm{length}(M_{\mathrm{tor}}) \leq \mathrm{length}(\mathrm{F}^i_{\mathrm{tor}}) + \mathrm{length}((M/\mathrm{F}^i)_{\mathrm{tor}}),
\]
with equality holds if and only if the map
\[
M_{\mathrm{tor}} \to (M/\mathrm{F}^i)_{\mathrm{tor}}
\]
is surjective.
Hence we see that the equality in condition is equivalent to
\[
M_{\mathrm{tor}} \to (M/\mathrm{F}^i)_{\mathrm{tor}}
\]
being surjective for all \(i\).

Applying the snake lemma to
\[
\xymatrix{
0 \ar[r] & \mathrm{F}^i_{\mathrm{tor}} \ar[d] \ar[r] & \mathrm{F}^i \ar[d] \ar[r] & \mathrm{F}^i_{\mathrm{tf}} \ar[d] \ar[r] & 0 \\
0 \ar[r] & M_{\mathrm{tor}}          \ar[r] & M          \ar[r] & M_{\mathrm{tf}} \ar[r] & 0
}
\]
yields an exact sequence
\[
0 \to M_{\mathrm{tor}}/\mathrm{F}^i_{\mathrm{tor}} \to (M/\mathrm{F}^i)_{\mathrm{tor}} 
\to (M_{\mathrm{tf}}/\mathrm{F}^i_{\mathrm{tf}})_{\mathrm{tor}} \to 0,
\]
here we used the fact that \(\mathrm{F}^i \cap M_{\mathrm{tor}} = \mathrm{F}^i_{\mathrm{tor}}\).
This short exact sequence says exactly that the surjectivity of
\[
M_{\mathrm{tor}} \to (M/\mathrm{F}^i)_{\mathrm{tor}}
\]
is equivalent to \(M_{\mathrm{tf}}/\mathrm{F}^i_{\mathrm{tf}}\) being torsion-free,
hence concludes the proof of this lemma.
\end{proof}
\end{proof}

Before proving the second part of~\cref{characterizing degenerations},
let us briefly discuss the condition of an extension of finitely generated torsion \(R\)-modules being split.

\begin{definition}
\label{char polygon defn}
Let \(M\) be a finitely generated torsion \(R\)-module.
Write 
\[
M = \oplus_{i = 1}^{l} R/\pi^{n_i}
\] 
where \(n_1 \leq n_2 \leq \ldots \leq n_l\).
Then the \emph{characteristic polygon} of \(M\), denoted by \(\mathcal{P}_M\), is
the graph of the piece-wise linear function defined on \([0,l]\) passing through \((0,0)\),
with the \(i\)-th segment of horizontal span \(1\) 
and slope \(n_i\).
\end{definition}

\begin{remark}
It is easy to see that the width of \(\mathcal{P}_M\) is given by \(\dim_{\kappa} (M/\pi M)\),
and the end points are given by \((0,0)\) and \((\dim_{\kappa} (M/\pi M), \mathrm{length}(M))\).
\end{remark}

Given two finitely generated torsion \(R\)-modules, 
we would like to compare the characteristic polygons of an extension class and that of their direct sums.

\begin{example}[{see also~\cite[P.~502]{dJ93}}]
\label{guiding example}
Consider an extension:
\[
0 \to N = R/\pi^l \to M \to R/\pi^m \to 0,
\]
then we must have either \(M \simeq R/\pi^{l+m}\), or 
\(M \simeq R/\pi^n \oplus R/\pi^{m+l-n}\) where \(\min \{n,m+l-n \} \leq \min \{l,m \} \)
with equality if and only if the short exact sequence splits.

We observe the former case corresponds to \(N/\pi \to M/\pi\) being not injective.
In the latter case we see that \(\mathcal{P}_M\)
is always lower than or equal to \(\mathcal{P}_{N \oplus M/N}\), 
and equality holds exactly when the extension class splits.
\end{example}

Here by $\mathcal{P}_{N \oplus M/N}$ we mean the characteristic polygon associated with the
module $N \oplus M/N$.
Inspired by this example, 
we give the following criterion characterizing split short exact sequences of finitely generated torsion \(R\)-modules.

\begin{proposition}
\label{criterion of split SES}
Let \(M\) be a finitely generated torsion \(R\)-module, and \(N \subset M\) is a submodule.
Suppose \(N/\pi \to M/\pi\) is an injection.
Then \(\mathcal{P}_M\) is lower than or equal to \(\mathcal{P}_{N \oplus M/N}\),
with equality holds if and only if \(N \subset M\) splits.
\end{proposition}

\begin{proof}
First assume we can prove the statement when \(N\) is cyclic (generated by one element).
Write \(N = N_1 \oplus N_2\), inducting on the dimension of \(N/\pi\),
we see that 
\[
\mathcal{P}_M \leq \mathcal{P}_{N_1 \oplus M/N_1} 
\leq \mathcal{P}_{N_1 \oplus N_2 \oplus M/N} = \mathcal{P}_{N \oplus M/N}
\]
with equality holds if and only if both \(N_1 \subset M\) and \(N_2 \subset M/N_1\)
split, or equivalently \(N = N_1 \oplus N_2 \subset M\) splits.
Therefore we reduce to the case where \(N = R/\pi^n\) is cyclic.

Dually, we may induct on the dimension of \((M/N)/\pi\).
By the same argument as above, we may also assume that \(M/N\) is also cyclic.
Now we have reduced the statement to the case where both of \(N\) and \(M/N\) are cyclic,
which is discussed in~\cref{guiding example}.
\end{proof}

We may extend this discussion to a multi-filtered situation,
which is useful in considerations of spectral sequences.
Below let us record one consequence of~\cref{criterion of split SES}.

\begin{corollary}
\label{split multifiltration corollary}
Let \((M, \mathrm{F}^{\bullet})\) be a finitely generated torsion \(R\)-module
with an exhaustive and complete filtration.
Suppose that we have an abstract isomorphism 
\[
M \simeq \bigoplus_i \mathrm{Gr}^i,
\]
then all of \(\mathrm{F}^i \subset M\) are direct summands.
\end{corollary}

\begin{proof}
Without loss of generality, let us assume that the filtration is increasing.
Now $M \simeq \bigoplus_i \mathrm{Gr}^i$ and is a Noetherian $R$-module,
and the filtration is exhaustive and complete,
we may assume that \(\mathrm{F}^i = 0 \) for \(i < 0\)
and that there is an integer $N$
such that $\mathrm{F}^j = M$ whenever $j > N$.

Next let us show that the natural map
$\mathrm{F}^i/\pi \to \mathrm{F}^{i+1}/\pi$ is injective for all $i$.
Due to right exactness of reduction modulo $\pi$, we have a chain of inequalities:
\[
\dim_{\kappa}(M/\pi) \leq \dim_{\kappa}(\mathrm{F}^{N-1}/\pi) + \dim_{\kappa}(\mathrm{Gr}^{N}/\pi)
\leq \ldots 
\]
\[
\leq \dim_{\kappa}(\mathrm{F}^{i}/\pi) + \sum_{j = i+1}^N \dim_{\kappa}(\mathrm{Gr}^j/\pi)
\leq \ldots \leq \sum_{j = 0}^N \dim_{\kappa}(\mathrm{Gr}^j/\pi) = \dim_{\kappa}(M/\pi),
\]
where the last equality follows from the condition that $M \simeq \bigoplus_i \mathrm{Gr}^i$.
Therefore all the inequalities in the above chain must in fact be equalities,
which is equivalent to saying all the maps
$\mathrm{F}^i/\pi \to \mathrm{F}^{i+1}/\pi$ are injective.

Now~\cref{criterion of split SES} implies that
\[
\mathcal{P}_M \leq \mathcal{P}_{\mathrm{Gr}^0 \oplus M/\mathrm{F}^0}
\leq \mathcal{P}_{\mathrm{Gr}^0 \oplus \mathrm{Gr}^1 \oplus M/\mathrm{F}^1}
\leq \ldots \leq \mathcal{P}_{\bigoplus_i \mathrm{Gr}^i},
\]
and our condition forces all the inequalities above to be an equality.
Hence applying~\cref{criterion of split SES} again yields what we want.
\end{proof}

Now we turn to the proof of (the ``if'' part of)~\cref{characterizing degenerations}(2). 

\begin{proof}[Proof of~\cref{characterizing degenerations}(2)]
By validity of (1), we see that in this situation, the spectral sequence is already saturated degenerate.
Therefore it suffices to show that the induced filtration 
\(H^i(\mathrm{Fil}^j C)_{\mathrm{tor}}= \mathrm{Fil}^j H^i(C)_{\mathrm{tor}}\) on
\(H^i(C)_{\mathrm{tor}}\) is split for all \(i\).

Notice that in our proof of (1), we established that the graded pieces of this filtration is exactly
given by \(H^i(\mathrm{Gr}^j C)_{\mathrm{tor}}\).
Now our condition implies that we have an abstract isomorphism
\[
H^i(C)_{\mathrm{tor}} \simeq \bigoplus_{j} \mathrm{Gr}^j H^i(C)_{\mathrm{tor}}.
\]
Applying~\cref{split multifiltration corollary}, we see that
\(H^i(\mathrm{Fil}^j C)_{\mathrm{tor}} = \mathrm{Fil}^j H^i(C)_{\mathrm{tor}}
\subset H^i(C)_{\mathrm{tor}}\) is split for all \(i\).
\end{proof}

Inspired by~\cref{characterizing degenerations}, we make the following further definition
concerning torsion part of various pages of the spectral sequence.

\begin{definition}
\label{definition of torsion degeneration}
\leavevmode
\begin{enumerate}
    \item We say the spectral sequence \emph{has saturated degenerate torsion in degree $i$}
    if we have an equality
    \[
    \mathrm{length}(H^i(C)_{\mathrm{tor}}) = \sum_{j} \mathrm{length}(H^i(\mathrm{Gr}^j C)_{\mathrm{tor}})
    \]
    \item We say the spectral sequence \emph{has split degenerate torsion in degree $i$} 
    if we have an abstract isomorphism of \(R\)-modules:
    \[
    H^i(C)_{\mathrm{tor}} \simeq \bigoplus_{j} H^i(\mathrm{Gr}^j C)_{\mathrm{tor}}.
    \]
\end{enumerate}
\end{definition}

The following Proposition is similar to~\cref{characterizing degenerations}.

\begin{proposition}
\label{characterizing torsion degenerations}
Suppose that the spectral sequence $(C, \mathrm{Fil}^\bullet)$ degenerates after inverting \(\pi\).
Let $i$ be an integer.
\begin{enumerate}
    \item If the spectral sequence has saturated degenerate torsion in degree $i$,
    then we have identifications:
    \[
    H^i(\mathrm{Gr}^j)_{\mathrm{tor}} \cong \mathrm{Gr}^j(H^i(C))_{\mathrm{tor}}.
    \]
    Consequently the maps $H^i(\mathrm{Fil}^j) \to H^i(C)$ are injective
    and the induced injection $H^i(\mathrm{Fil}^j)_{\mathrm{tf}} \to H^i(C)_{\mathrm{tf}}$
    is saturated  for all $j$.
    \item If the spectral sequence has split degenerate torsion in degree $i$,
    then the induced maps $H^i(\mathrm{Fil}^j) \to H^i(C)$ are split for all $j$.
\end{enumerate}
\end{proposition}

\begin{proof}
For (1): in the proof of~\cref{characterizing degenerations} (1) we have established that
the equality of lengths of torsions implies the claimed identifications.
Since the spectral sequence is assumed to be degenerate after inverting $\pi$,
we immediately get that the maps $H^i(\mathrm{Fil}^j) \to H^i(C)$ are injective.

Next we look at the following diagram:
\[
\xymatrix{
0 \ar[r] & H^i(\mathrm{Fil}^j)_{\mathrm{tor}} \ar[r] \ar[d] & H^i(C)_{\mathrm{tor}} \ar[r] \ar[d] &
H^i(C/\mathrm{Fil}^j)_{\mathrm{tor}} \ar[d] \ar[r] & 0 \\
0 \ar[r] & H^i(\mathrm{Fil}^j) \ar[r] \ar[d] & H^i(C) \ar[r] \ar[d] & H^i(C/\mathrm{Fil}^j) \ar[d] & \\
& H^i(\mathrm{Fil}^j)_{\mathrm{tf}} \ar[r] & H^i(C)_{\mathrm{tf}} \ar[r] & H^i(C/\mathrm{Fil}^j)_{\mathrm{tf}} &
}.
\]
Notice that the first two rows are exact, and the snake lemma gives us an exact sequence
\[
0 \to H^i(\mathrm{Fil}^j)_{\mathrm{tf}} \to H^i(C)_{\mathrm{tf}} \to H^i(C/\mathrm{Fil}^j)_{\mathrm{tf}},
\]
which implies that $H^i(C)_{\mathrm{tf}}/H^i(\mathrm{Fil}^j)_{\mathrm{tf}}$, 
as a submodule of $H^i(C/\mathrm{Fil}^j)_{\mathrm{tf}}$, is torsion-free.

For (2): using what we have proved in (1), all we need to show is that the inclusion
$H^i(\mathrm{Fil}^j)_{\mathrm{tor}} \to H^i(C)_{\mathrm{tor}}$ splits,
which follows from identifications $H^i(\mathrm{Gr}^j)_{\mathrm{tor}} \cong \mathrm{Gr}^j(H^i(C))_{\mathrm{tor}}$
and~\cref{split multifiltration corollary}.
\end{proof}

One can also define what it means for the spectral sequence to have saturated/split degenerate
torsion in a range of degrees.
The above~\cref{characterizing torsion degenerations} is immediately generalized
to this generality.

The following~\cref{reading off dimension from special fiber} will be used in later sections.

\begin{proposition}
\label{reading off dimension from special fiber}
Let \(M\) be a finitely generated \(R\)-module, with \(N \subset M\) a submodule.
Suppose that \(N_{\mathrm{tf}} \subset M_{\mathrm{tf}}\) is saturated.
Then we have equality of dimensions:
\[
\dim_K N[1/\pi] = \dim_{\kappa} \mathrm{Im}(N/\pi \to M_{\mathrm{tf}}/\pi).
\]
\end{proposition}

\begin{proof}
Both sides do not change when we passing from \(M\) to \(M_{\mathrm{tf}}\).
Hence we may assume \(M\) to be torsion-free, 
in which case \(N\) is a direct summand by being a saturated submodule,
and the statement becomes trivial in this situation.
\end{proof}

The next two Lemmas are pointed out by the anonymous referee, we thank them for this suggestion.

\begin{lemma}
\label{identification of sub}
Let $U \to V$ be a map of $R$-complexes.
Let $n$ be an integer and assume that $H^{n+1}(U) \to H^{n+1}(V)$ is injective.
Then we have an identification of subspaces in $H^n(V)/\pi$:
\[
\xymatrix{
\mathrm{Im}(H^n(U)/\pi \to H^n(V)/\pi)
\ar[d]^{\cong} \\
\mathrm{Im}(H^n(U/\pi) \to H^n(V/\pi)) 
\bigcap H^n(V)/\pi.
}
\]
\end{lemma}

\begin{proof}
Let us denote the cone of $U \to V$ by $C$.
The assumption yields an exact sequence
\[
H^n(U) \to H^n(V)
\to H^n(C) \to 0.
\]
Modulo \(\pi\), and use right exactness, we get
\[
H^n(U)/\pi \to H^n(V)/\pi
\to H^n(C)/\pi \to 0.
\]
Now we have the following diagram with horizontal lines being exact and vertical arrows being injective:
\[
\xymatrix{
H^n(U)/\pi \ar[r] \ar@{^{(}->}[d] & 
H^n(V)/\pi \ar[r] \ar@{^{(}->}[d] &
H^n(C)/\pi \ar@{^{(}->}[d]\\
H^n(U/\pi) \ar[r] &
H^n(V/\pi) \ar[r] &
H^n(C/\pi).
}
\]
A simple diagram chasing now gives the claimed identification.
\end{proof}

Combining~\cref{reading off dimension from special fiber} and~\cref{identification of sub} gives us
the following

\begin{lemma}
\label{referee suggest}
Let $U \to V$ be a map of $R$-complexes.
Let $n$ be an integer.
Assume that
\begin{enumerate}
    \item the image of $H^n(U) \to H^n(V)_{\mathrm{tf}}$ is saturated inside $H^n(V)_{\mathrm{tf}}$; and
    \item the map $H^{n+1}(U) \to H^{n+1}(V)$ is injective.
\end{enumerate}
Then we have an equality of dimensions:
\[
\dim_K \mathrm{Im}(H^n(U)[1/\pi] \to H^n(V)[1/\pi]) =
\]
\[
\dim_{\kappa} \mathrm{Im}(\Big(\mathrm{Im}\big(H^n(U/\pi) \to H^n(V/\pi)\big) 
\bigcap H^n(V)/\pi\Big) \to H^n(V)_{\mathrm{tf}}/\pi).
\]
\end{lemma}

\begin{proof}
We use~\cref{reading off dimension from special fiber}, where $N \subset M$ are
given by the image of $H^n(U)$ inside $H^n(V)$.
We then use~\cref{identification of sub} to rewrite the right hand side of the equality.
\end{proof}

For our purpose later, we also need to discuss the relation of cohomologies
of a perfect $R$-complex and cohomologies of its dual.

\begin{lemma}
\label{duality for perfect complex}
Let $U$ be a perfect $R$-complex. Consider $V = RHom_R(U, R[0])$.
Then for any integer $i$ we have canonical identifications:
\[
H^{-i+1}(V)_{\mathrm{tor}} \cong \mathrm{Hom}_R(H^i(U)_{\mathrm{tor}}, K/R)
\text{ and } H^{-i}(V)_{\mathrm{tf}} \cong \mathrm{Hom}_R(H^i(U)_{\mathrm{tf}}, R).
\]
\end{lemma}

\begin{proof}
We have an $E_2$ spectral sequence:
\[
\mathrm{Ext}^i_R(H^j(U), R) \Longrightarrow H^{i-j}(V).
\]
Since $R$ is a DVR, the terms on second page vanish unless $i \in \{0, 1\}$.
Hence the spectral sequence degenerates for degree reason, and we get
a natural short exact sequence 
\[
0 \to \mathrm{Ext}^1_R(H^i(U), R) \to H^{-i+1}(V) \to \mathrm{Hom}_R(H^{i-1}(U), R) \to 0.
\]
Now we finish by recalling that given a finitely generated $R$-module $M$, one has
canonical identifications:
\[
\mathrm{Hom}_R(M, R) = \mathrm{Hom}_R(M_{\mathrm{tf}}, R)
\text{ and }
\mathrm{Ext}^1_R(M, R) = \mathrm{Hom}_R(M_{\mathrm{tor}}, K/R).
\]
\end{proof}

% * Consequences
\section{Consequences of recent developments in integral \(p\)-adic Hodge theory}
\label{Consequences}

In this section we concentrate ourselves in the \(p\)-adic situation.

\subsection{Notations and Setup}
\label{notations}
Throughout this section, 
let \(K\) be a complete \(p\)-adic field with a chosen uniformizer \(\pi\), 
ring of integers \(\mathcal{O}_K\)
and perfect residue field \(\kappa \coloneqq \mathcal{O}_K/(\pi)\).
Recall that \(\mathcal{O}_K\) contains the ring of Witt vectors of \(\kappa\),
and the degree of their fraction field extension \(K_0 \coloneqq \mathrm{W}(\kappa)[1/p] \subset K\)
is called ramification index of \(K\) and denoted by \(e\).
This paper concerns low ramification situation, in particular we assume that \(e \leq p-1\),
therefore the ideal \((\pi) \subset \mathcal{O}_K\) has a unique divided power structure.

Denote \(\mathfrak{S} \coloneqq \mathrm{W}(\kappa)[\![u]\!]\) with a surjection
\(\mathfrak{S} \twoheadrightarrow \mathcal{O}_K\), 
where \(u\) is sent to the chosen uniformizer \(\pi\).
The kernel of this surjection is generated by an Eisenstein polynomial \(I = (E(u))\).
Define the Frobenius \(\phi \colon \mathfrak{S} \to \mathfrak{S}\) that extends
the Frobenius on \(\mathrm{W}(\kappa)\) and sends \(u\) to \(u^p\),
since \(\mathfrak{S}\) is \(p\)-torsion free, 
this puts a unique \(\delta\)-structure on \(\mathfrak{S}\).
The pair \((\mathfrak{S},I)\) is a Breuil--Kisin type prism (see~\cite[Example 1.3 (3)]{BS19}).

Let \(\mathcal{X}\) be a smooth proper formal scheme over \(\mathcal{O}_K\).
We denote the special fiber of \(\mathcal{X}\) by 
\(\mathcal{X}_0 \coloneqq \mathcal{X} \times_{\mathcal{O}_K} \kappa\)
and the (rigid) generic fiber of \(\mathcal{X}\) by 
\(X \coloneqq \mathcal{X} \times_{\mathcal{O}_K} K\).
We call \(\mathcal{X}\) a \emph{lifting of \(\mathcal{X}_0\) over \(\mathcal{O}_K\)}.
In the case where \(e=1\), i.e.~\(\mathcal{O}_K = \mathrm{W}(\kappa)\), 
we call \(\mathcal{X}\) an \emph{unramified lift of \(\mathcal{X}_0\)}.

\subsection{Virtual Hodge Numbers}

Recall that we have a natural identification 
\(\mathrm{R\Gamma_{crys}}(\mathcal{X}_0/\mathrm{W(\kappa)}) 
\otimes^{\mathbb{L}}_{\mathrm{W}(\kappa)} \kappa 
\simeq \mathrm{R\Gamma_{dR}}(\mathcal{X}_0/\kappa)\),
which implies that we have a natural injection
\[
H^i_{\mathrm{crys}}(\mathcal{X}_0/\mathrm{W(\kappa)})/p
\hookrightarrow H^i_{\mathrm{dR}}(\mathcal{X}_0/\kappa),
\]
for all \(i\).
Therefore we may regard 
\(H^i_{\mathrm{crys}}(\mathcal{X}_0/\mathrm{W(\kappa)})_{\mathrm{tf}}/p\)
as a natural subquotient of \(H^i_{\mathrm{dR}}(\mathcal{X}_0/\kappa)\).

\begin{definition}[virtual Hodge numbers]
\label{virtual Hodge}
The Hodge filtrations on \(H^i_{\mathrm{dR}}(\mathcal{X}_0/\kappa)\)
induces natural filtrations on the subquotient 
\(H^i_{\mathrm{crys}}(\mathcal{X}_0/\mathrm{W(\kappa)})_{\mathrm{tf}}/p\).
The virtual Hodge numbers of \(\mathcal{X}_0\), is given by
\[
\mathfrak{h}^{i,j}(\mathcal{X}_0) \coloneqq
\dim_{\kappa} \mathrm{Fil}^{i} (H^{i+j}_{\mathrm{crys}}(\mathcal{X}_0/\mathrm{W(\kappa)})_{\mathrm{tf}}/p)
- \dim_{\kappa} \mathrm{Fil}^{i+1} (H^{i+j}_{\mathrm{crys}}(\mathcal{X}_0/\mathrm{W(\kappa)})_{\mathrm{tf}}/p)
\]
\end{definition}

Unwinding the definition, we have the following description of the \(i\)-th induced filtration on
\(H^{n}_{\mathrm{crys}}(\mathcal{X}_0/\mathrm{W(\kappa)})_{\mathrm{tf}}/p \):
\[
\label{complicated filtration}
\tag{\epsdice{1}}
\mathrm{Im}(\mathrm{Im}(H^n(\Omega^{\geq i}_{\mathcal{X}_0/\kappa}) \to H^n(\Omega^{\bullet}_{\mathcal{X}_0/\kappa})) 
\bigcap H^n_{\mathrm{crys}}(\mathcal{X}_0/\mathrm{W(\kappa)})/p \to 
H^n_{\mathrm{crys}}(\mathcal{X}_0/\mathrm{W(\kappa)})_{\mathrm{tf}}/p).
\]
Note that this definition \emph{only depends} on the smooth proper variety \(\mathcal{X}_0\)
in characteristic \(p\).

\begin{remark}
\label{identification remark}
It is worth pointing out that in this definition, 
we may replace the Witt vectors by any ring of integers \(\mathcal{O}_K\) as long as the ramification index \(e \leq p-1\).
By the virtue of the de Rham--crystalline comparison and the
base change formula of crystalline cohomology~\cite[Corollary 7.3 and Theorem 7.8]{BOgus},
we have natural identifications:
\begin{itemize}
\item \(
H^i_{\mathrm{crys}}(\mathcal{X}_0/\mathrm{W(\kappa)}) \otimes_{\mathrm{W(\kappa)}} \mathcal{O}_K
\cong H^i_{\mathrm{dR}}(\mathcal{X}/\mathcal{O}_K)
\),
\item \(
H^i_{\mathrm{crys}}(\mathcal{X}_0/\mathrm{W(\kappa)})_{\mathrm{tf}} \otimes_{\mathrm{W(\kappa)}} \mathcal{O}_K
\cong H^i_{\mathrm{dR}}(\mathcal{X}_0/\mathcal{O}_K)_{\mathrm{tf}}
\),
\item \(
H^i_{\mathrm{crys}}(\mathcal{X}_0/\mathrm{W(\kappa)})/p 
\cong H^i_{\mathrm{dR}}(\mathcal{X}_0/\mathcal{O}_K)/\pi
\), and
\item \(
H^i_{\mathrm{crys}}(\mathcal{X}_0/\mathrm{W(\kappa)})_{\mathrm{tf}}/p
\cong H^i_{\mathrm{dR}}(\mathcal{X}_0/\mathcal{O}_K)_{\mathrm{tf}}/\pi
\).
\end{itemize}
\end{remark}

These filtrations are only objects in characteristic \(p\).
The goal of this subsection is to show that if the integral \(p\)-adic Hodge filtration behaves nicely,
then these filtrations can tell us something about the rational Hodge filtrations.

\begin{proposition}
\label{identification of filtrations}
Let \(\mathcal{X} \to \mathrm{Spf}(\mathcal{O}_K)\) be as in Subsection~\ref{notations}.
Assume that the integral Hodge--de Rham spectral sequence of \(\mathcal{X}\) degenerates,
then we have identifications of subspaces in
\(H^n(\Omega^{\bullet}_{\mathcal{X}/\mathcal{O}_K})/\pi \cong H^n_{\mathrm{crys}}(\mathcal{X}_0/\mathrm{W(\kappa)})/p\):
\[
\xymatrix{
\mathrm{Im}(H^n(\Omega^{\geq i}_{\mathcal{X}/\mathcal{O}_K})/\pi \to H^n(\Omega^{\bullet}_{\mathcal{X}/\mathcal{O}_K})/\pi)
\ar[d]^{\cong} \\
\mathrm{Im}(H^n(\Omega^{\geq i}_{\mathcal{X}_0/\kappa}) \to H^n(\Omega^{\bullet}_{\mathcal{X}_0/\kappa})) 
\bigcap H^n_{\mathrm{crys}}(\mathcal{X}_0/\mathrm{W(\kappa)})/p.
}
\]
\end{proposition}

\begin{proof}
This is a direct application of~\cref{identification of sub}.
We write $U = \mathrm{R\Gamma}(\Omega^{\geq i}_{\mathcal{X}/\mathcal{O}_K})$
and $V = \mathrm{R\Gamma}(\Omega^{\bullet}_{\mathcal{X}/\mathcal{O}_K})$.
Then the cone $C = \mathrm{R\Gamma}(\Omega^{\leq i-1}_{\mathcal{X}/\mathcal{O}_K})$.
The assumption on Hodge--de Rham spectral sequence implies the condition required in~\cref{identification of sub}.
%\[
%0 \to H^n(\Omega^{\geq i}_{\mathcal{X}/\mathcal{O}_K}) \to H^n(\Omega^{\bullet}_{\mathcal{X}/\mathcal{O}_K})
%\to H^n(\Omega^{\leq i-1}_{\mathcal{X}/\mathcal{O}_K}) \to 0.
%\]
%Modulo \(\pi\), we get
%\[
%H^n(\Omega^{\geq i}_{\mathcal{X}/\mathcal{O}_K})/\pi \to H^n(\Omega^{\bullet}_{\mathcal{X}/\mathcal{O}_K})/\pi
%\to H^n(\Omega^{\leq i-1}_{\mathcal{X}/\mathcal{O}_K})/\pi \to 0.
%\]
%Now we have the following diagram with horizontal lines being exact and vertical arrows being injective:
%\[
%\xymatrix{
%H^n(\Omega^{\geq i}_{\mathcal{X}/\mathcal{O}_K})/\pi \ar[r] \ar@{^{(}->}[d] & 
%H^n(\Omega^{\bullet}_{\mathcal{X}/\mathcal{O}_K})/\pi \ar[r] \ar@{^{(}->}[d] &
%H^n(\Omega^{\leq i-1}_{\mathcal{X}/\mathcal{O}_K})/\pi \ar@{^{(}->}[d]\\
%H^n(\Omega^{\geq i}_{\mathcal{X}_0/\kappa}) \ar[r] &
%H^n(\Omega^{\bullet}_{\mathcal{X}_0/\kappa}) \ar[r] &
%H^n(\Omega^{\leq i-1}_{\mathcal{X}_0/\kappa}).
%}
%\]
%A simple diagram chasing, 
Together with the identification spelled out in~\cref{identification remark},
we get the claimed identification.
\end{proof}

Recall the definition of a spectral sequence being saturated degenerate in~\cref{definition of SS degeneration}.
The following is an immediate consequence of~\cref{referee suggest} applying to all the Hodge filtrations
of the de Rham complex,
but let us repeat the proof one more time.

\begin{proposition}
\label{equality of numbers}
Let \(\mathcal{X} \to \mathrm{Spf}(\mathcal{O}_K)\) be as in Subsection~\ref{notations}.
Assume that the integral Hodge--de Rham spectral sequence of \(\mathcal{X}\) is saturated degenerate,
then we have equality of (virtual) Hodge numbers:
\[
\mathfrak{h}^{i,j}(\mathcal{X}_0) = h^{i,j}(X).
\]
\end{proposition}

\begin{proof}
According to the definitions, 
we need to show that the dimension of \(i\)-th filtration on 
\(H^n_{\mathrm{crys}}(\mathcal{X}_0/\mathrm{W(\kappa)})_{\mathrm{tf}}/p
\cong H^n_{\mathrm{crys}}(\mathcal{X}_0/\mathcal{O}_K)_{\mathrm{tf}}/\pi\)
agrees with the dimension of \(i\)-th filtration on \(H^n_{\mathrm{dR}}(X)\).

By~\cref{identification of filtrations}, 
we may rewrite the formula~\ref{complicated filtration} of the \(i\)-th filtration on 
\(H^n_{\mathrm{crys}}(\mathcal{X}_0/\mathcal{O}_K)_{\mathrm{tf}}/\pi
\cong H^n(\Omega^{\bullet}_{\mathcal{X}/\mathcal{O}_K})_{\mathrm{tf}}/\pi\) by
\[
\mathrm{Im}\Big(H^n(\Omega^{\geq i}_{\mathcal{X}/\mathcal{O}_K})/\pi \to H^n(\Omega^{\bullet}_{\mathcal{X}/\mathcal{O}_K})/\pi
\to H^n(\Omega^{\bullet}_{\mathcal{X}/\mathcal{O}_K})_{\mathrm{tf}}/\pi \Big).
%\subset H^n_{\mathrm{crys}}(\mathcal{X}_0/\mathcal{O}_K)_{\mathrm{tf}}/\pi.
\]
Our assumption implies that the submodule
\(H^n(\Omega^{\geq i}_{\mathcal{X}/\mathcal{O}_K}) 
\subset H^n(\Omega^{\bullet}_{\mathcal{X}/\mathcal{O}_K})\)
meets the condition of~\cref{reading off dimension from special fiber}.
Hence we may apply~\cref{reading off dimension from special fiber} 
which gives us the claimed equality of dimensions of filtrations.
\end{proof}

\subsection{Main Theorem}

In this subsection, we explain the proof of the Main~\cref{main theorem},
which we repeat below:

\begin{theorem}
Let \(\mathcal{X} \to \mathrm{Spf}(\mathcal{O}_K)\) be a smooth proper formal scheme.
Assume that
\begin{enumerate}
    \item there is a lift of \(\mathcal{X}\) over \(\mathfrak{S}/(E^2)\); and
    \item the relative dimension of \(\mathcal{X}\) 
    and the ramification index satisfy the inequality: 
    \((\dim \mathcal{X}_0 + 1) \cdot e < p-1\).
\end{enumerate}
Then the Hodge--de Rham spectral sequence for \(\mathcal{X}\) is split degenerate.
In particular, we have equality of (virtual) Hodge numbers:
\[
\mathfrak{h}^{i,j}(\mathcal{X}_0) = h^{i,j}(X).
\]
\end{theorem}

\begin{remark}
\leavevmode
\begin{enumerate}
\item After modulo \((u)\), the surjection 
\(\mathfrak{S}/(E^2) \twoheadrightarrow \mathcal{O}_K\)
becomes \(\mathrm{W}_2(\kappa) \twoheadrightarrow \kappa\)
(note that \(E\) is an Eisenstein polynomial in \(u\)).
So we may view this Theorem as a mixed-characteristic analogue of a theorem by 
Deligne--Illusie~\cite[Corollaire 2.4]{DI87}.
\item Our result implies that a smooth proper variety
\(\mathcal{X}_0\) in positive characteristic knows Hodge numbers of the generic fiber
of a (formal) lifting provided:
(i) the lifting can be further lifted to \(\mathfrak{S}/(E^2)\); and
(ii) we have an inequality \((\dim \mathcal{X}_0 + 1) \cdot e < p-1\).
\item The condition~(1) is not so easy to verify.
There are two cases we can think of in which this condition is automatically guaranteed.
The first being that \(\mathcal{X}\) is an unramified lift 
(for then the surjection 
\(\mathfrak{S} \twoheadrightarrow \mathcal{O}_K\) admits a section), 
in which case a stronger statement follows from the work of 
Fontaine--Messing~\cite{FM87} or Kato~\cite{K87}, see~\cref{FM remark}.
The second case is when \(\mathcal{X}_0\) has unobstructed deformation theory, 
e.g.~when \(H^2(\mathcal{X}_0, \mathrm{T}) = 0\), by deformation theoretic considerations.
\item Similar to the situation of Deligne--Illusie's statement,
there are examples showing the necessity of condition~(1).
In fact one example comes from (ramified) liftings of 
(counter-)examples to Deligne--Illusie's statement in characteristic \(p\)
due to Antieau--Bhatt--Mathew~\cite{ABM19}, see~\cref{lifting ABM}
and more precisely~\cref{main example}.
\end{enumerate}
\end{remark}

The main ingredients that go into the proof are some
recent developments in
integral \(p\)-adic Hodge theory. 
So we shall postpone the proof until after the introduction of these ingredients.

Recently Bhatt--Scholze~\cite{BS19} developed prismatic cohomology theory
to unite many (if not all) known \(p\)-adic cohomology theories one may attach to
a \(p\)-adic smooth formal scheme over a certain class of \(p\)-adic base rings.
While their theory is in a much broader context,
we shall specialize their results to the situation of interest for this paper:
they introduced prismatic site \((\mathcal{X}/\mathfrak{S})_{\Prism}\)
on which there are two structure sheaves \(\mathcal{O}_{\Prism}\)
and \(\overline{\mathcal{O}}_{\Prism} \coloneqq
\mathcal{O}_{\Prism}/E\).
The cohomology of \(\mathcal{O}_{\Prism}\) on \((\mathcal{X}/\mathfrak{S})_{\Prism}\)
is denoted by \(\mathrm{R\Gamma_{\Prism}}(\mathcal{X}/\mathfrak{S})\).
There is a natural map of ringed topoi (see~\cite[Construction 4.4]{BS19}):
\[
\nu \colon \mathrm{Shv}((\mathcal{X}/\mathfrak{S})_{\Prism}, \overline{\mathcal{O}}_{\Prism})
\to \mathrm{Shv}(\mathcal{X}_{\mathrm{\acute{e}t}}, \mathcal{O}_{\mathcal{X}}).
\]
The properties we need of these objects are summarized in the following:

\begin{theorem}[Bhatt--Scholze]
\label{prisms}
\leavevmode
\begin{enumerate}
    \item (see~\cite[Corollary 15.4]{BS19}) 
    There is a canonical isomorphism
    \[
    \mathrm{R\Gamma_{dR}}(\mathcal{X}/\mathcal{O}_K) \cong
    \mathrm{R\Gamma_{\Prism}}(\mathcal{X}/\mathfrak{S})
    \otimes^{\mathbb{L}}_{\mathfrak{S}} \phi_*\mathcal{O}_K,
    \]
    where \(\phi_*\mathcal{O}_K\) is \(\mathcal{O}_K\)
    viewed as an \(\mathfrak{S}\)-module (or even algebra) via the composite
    \(\mathfrak{S} \xrightarrow{\phi} \mathfrak{S} \to 
    \mathcal{O}_K = \mathfrak{S}/(E)\).
    \item (see~\cite[Theorem 4.10]{BS19})
    There is a canonical isomorphism
    \[
    \Omega^i_{\mathcal{X}/\mathcal{O}_K}\{-i\} \cong
    R^i\nu_*\overline{\mathcal{O}}_{\Prism}.
    \]
    Here \((-)\{j\} \coloneqq (-) \otimes_{\mathcal{O}_K} ((E)/(E^2))^{\otimes j}\).
    This induces an increasing filtration (called the conjugate filtration) on 
    \(\mathrm{R\Gamma_{\Prism}}(\mathcal{X}/\mathfrak{S})
    \otimes^{\mathbb{L}}_{\mathfrak{S}} \mathcal{O}_K\) giving rise to
    an \(E_2\) spectral sequence (called the Hodge--Tate spectral sequence):
    \[
    \label{HT SS}
    \tag{\epsdice{2}}
    E_2^{i,j} = 
    H^i(\mathcal{X}, \Omega^j_{\mathcal{X}/\mathcal{O}_K})\{-j\}
    \Longrightarrow H^{i+j}_{\mathrm{HT}}(\mathcal{X}/\mathcal{O}_K) \coloneqq 
    H^{i+j}(\mathrm{R\Gamma_{\Prism}}(\mathcal{X}/\mathfrak{S})
    \otimes^{\mathbb{L}}_{\mathfrak{S}} \mathcal{O}_K).
    \]
    \item (see~\cite[Remark 4.13 and Proposition 4.14]{BS19}, c.f.~\cite[Proposition 3.2.1]{ALB19})
    The map 
    \[
    \mathcal{O}_{\mathcal{X}} \to 
    \tau^{\leq 1}R\nu_*\overline{\mathcal{O}}_{\Prism}
    \]
    splits if and only if \(\mathcal{X}\) lifts to \(\mathfrak{S}/(E^2)\).
\end{enumerate}
\end{theorem}

Lastly we need a result of Min~\cite{Min19} concerning the \(\mathfrak{S}\)-module structure
of \(H^i_{\Prism}(\mathcal{X}/\mathfrak{S})\) in
the case of small \(i\) and low ramification index:

\begin{theorem}[{Min~\cite[Theorem 5.8]{Min19}}]
\label{structure of prismatic}
When \(i \cdot e < p-1\), we have an abstract isomorphism of \(\mathfrak{S}\)-modules:
\[
H^i_{\Prism}(\mathcal{X}/\mathfrak{S}) \cong
\mathfrak{S}^{n} \oplus \bigoplus_{j \in J} \mathfrak{S}/p^{n_j},
\]
where \(J\) is a finite set.
\end{theorem}

Using the result of Min and our analysis of spectral sequences in the previous section,
we may relate the behavior of the Hodge--de Rham and the Hodge--Tate spectral sequences.
Let $T$ be the largest integer with $T \cdot e < p-1$, which is the threshold given by
Min's Theorem above.

\begin{corollary}
\label{equivalence of SS}
Let \(i \) be an integer which is strictly smaller than $T$, we have two equivalences:
\begin{enumerate}
    \item The Hodge--de Rham spectral sequence having saturated degenerate torsion in degree $i$
    is equivalent to the Hodge--Tate spectral sequence having saturated degenerate torsion in degree $i$; and
    \item The Hodge--de Rham spectral sequence having split degenerate torsion in degree $i$
    is equivalent to the Hodge--Tate spectral sequence having split degenerate torsion in degree $i$.
\end{enumerate}
\end{corollary}

\begin{proof}
Since \(\mathcal{O}_K\) and \(\phi_* \mathcal{O}_K\) are of flat dimension $1$
as \(\mathfrak{S}\)-modules, we have short exact sequences:
\[
0 \to H^i_{\Prism}(\mathcal{X}/\mathfrak{S}) \otimes_{\mathfrak{S}} \mathcal{O}_K
\to H^i(\mathrm{R\Gamma_{\Prism}}(\mathcal{X}/\mathfrak{S}) \otimes^{\mathbb{L}}_{\mathfrak{S}} \mathcal{O}_K)
\to Tor_1^{\mathfrak{S}}(H^{i+1}_{\Prism}(\mathcal{X}/\mathfrak{S}), \mathcal{O}_K)
\to 0,
\]
\[
0 \to H^i_{\Prism}(\mathcal{X}/\mathfrak{S}) \otimes_{\mathfrak{S}} \phi_* \mathcal{O}_K
\to H^i(\mathrm{R\Gamma_{\Prism}}(\mathcal{X}/\mathfrak{S}) \otimes^{\mathbb{L}}_{\mathfrak{S}} \phi_* \mathcal{O}_K)
\to Tor_1^{\mathfrak{S}}(H^{i+1}_{\Prism}(\mathcal{X}/\mathfrak{S}), \phi_* \mathcal{O}_K)
\to 0.
\]

Notice that \(\mathfrak{S}\) and \(\mathfrak{S}/p^\ell\) are, as \(\mathfrak{S}\)-modules,
Tor-independent with \(\mathcal{O}_K\) and \(\phi_* \mathcal{O}_K\).
Therefore in the situation considered, by~\cref{prisms} (1) and (2) 
and~\cref{structure of prismatic},
we have an abstract isomorphism of \(\mathcal{O}_K\)-modules:
\[
H^i_{\mathrm{dR}}(\mathcal{X}/\mathcal{O}_K) \simeq 
H^i_{\mathrm{HT}}(\mathcal{X}/\mathcal{O}_K)
\]
for all \(i < T\).

Moreover we know that both Hodge--de Rham and Hodge--Tate 
spectral sequences degenerate after inverting \(\pi\)
(see~\cite[Corollary 1.8]{Sch13} and~\cite[Theorem 1.7]{BMS1}).
Hence we reduce these two statements respectively 
to~\cref{definition of torsion degeneration} (1) and (2),
by observing that the starting pages of these two spectral sequences
are formed by abstractly isomorphic \(\mathcal{O}_K\) modules (after switching bi-degrees).
\end{proof}

Finally we are ready to prove~\cref{main theorem}.
\begin{proof}[Proof of~\cref{main theorem}]
By anti-symmetrizing (c.f.~\cite[step (a) of Theorem 2.1]{DI87}) the section
\[
\Omega^1_{\mathcal{O}_{\mathcal{X}}}\{-1\}[-1] \to
\tau^{\leq 1}R\nu_*\overline{\mathcal{O}}_{\Prism}
\]
given by (see~\cref{prisms} (3)) 
the lift of \(\mathcal{X}\) to \(\mathfrak{S}/(E^2)\),
we see that the conjugate filtration splits 
(note that our constraint on \(\dim X\) in particular
implies the relative dimension is smaller than \(p\)):
\[
\bigoplus_{i = 0}^{\dim X} \Omega^i_{\mathcal{O}_{\mathcal{X}}}\{-i\}[-i]
\simeq R\nu_*\overline{\mathcal{O}}_{\Prism}.
\]
Therefore we see that the Hodge--Tate spectral sequence~\ref{HT SS}
has split degenerate torsion in all degrees, 
and thus in particular in degrees $(-\infty, T-1]$.
By~\cref{equivalence of SS} we see that the Hodge--de Rham spectral sequence
must also have split degenerate torsion in the same range of degrees.

Now we invoke the duality statements for de Rham and Hodge cohomologies.
By Poincar\'{e} duality for de Rham cohomology~\cite[Chapter VII, Th\'{e}or\`{e}me 2.1.3]{Ber74} 
we have an identification
\[
\mathrm{R\Gamma_{dR}}(\mathcal{X}/\mathcal{O}_K) \cong 
\mathrm{RHom}_R(\mathrm{R\Gamma_{dR}}(\mathcal{X}/\mathcal{O}_K), \mathcal{O}_K[-2d]),
\]
where $d = \dim(\mathcal{X}_o)$ and we have implicitly identified de Rham and crystalline cohomologies.
Similarly, by Serre duality for Hodge cohomology~\cite[\href{https://stacks.math.columbia.edu/tag/0BRT}{Tag 0BRT}, \href{https://stacks.math.columbia.edu/tag/0AU3}{Tag 0AU3}, and \href{https://stacks.math.columbia.edu/tag/0E9Z}{Tag 0E9Z}]{stacks-project} we have an identification
\[
\mathrm{R\Gamma}(\Omega^{d-i}_{\mathcal{X}/\mathcal{O}_K}) \cong 
\mathrm{RHom}_R(\mathrm{R\Gamma}(\Omega^i_{\mathcal{X}/\mathcal{O}_K}), \mathcal{O}_K[-d]).
\]
Using~\cref{duality for perfect complex} and what we obtained in the previous paragraph,
we can now conclude that the Hodge--de Rham spectral sequence has split degenerate torsion
in degrees $[2d+2-T, +\infty)$.

The union of intervals $(-\infty, T-1] \bigcup [2d+2-T, +\infty)$ cover all integers
exactly when $d+1 \leq T$.
Therefore condition (2) now implies that the Hodge--de Rham spectral sequence
is split degenerate by~\cref{characterizing degenerations} (2).
The last statement concerning numerical equalities follows from~\cref{equality of numbers}.
\end{proof}

\begin{remark}
An ongoing project of Bhatt--Lurie establishes a duality statement for prismatic cohomology.
Assuming their result, together with Min's result, one get a more uniform proof of the Main Theorem.
Indeed these results together imply that under the assumption (2) all the prismatic cohomology
of $(\mathcal{X}/\mathfrak{S})_{\Prism}$ are of the shape 
$\mathfrak{S}^{n} \oplus \bigoplus_{j \in J} \mathfrak{S}/p^{n_j}$.
Notice however their duality statement does not improve the bound on dimension, because their duality statement
is over the \emph{$2$-dimensional} ring $\mathfrak{S}$.
\end{remark}

In the situation where \(\dim X\) exceeds the bound, our argument produces the following.
Recall that $T$ denotes the largest integer satisfying $T \cdot e < p-1$.

\begin{porism}
\label{porism}
Let \(\mathcal{X}\) be a smooth proper formal scheme over \(\mathcal{O}_K\)
which lifts to \(\mathfrak{S}/(E^2)\).
Then the differentials in the Hodge--de Rham spectral sequence 
with target of total degree \(\leq T - 1\) are zero,
the induced Hodge filtrations on de Rham cohomology of degree
\(\leq T - 1\) are split.
Hence we have \(\mathfrak{h}^{i,j}(\mathcal{X}_0) = H^{i,j}(X)\),
for \(i + j \leq T - 2\) (or equivalently 
\((i + j + 2) \cdot e < p - 1\)).
\end{porism}

\begin{proof}
%It is a theorem of Min~\cite[Theorem 0.1]{Min19}
%that in this situation, the prismatic cohomology \(H^a_{\Prism}(\mathcal{X}/\mathfrak{S})\)
%is isomorphic to the base change of the \(a\)-th \'{e}tale cohomology
%(of the geometric rigid generic fiber), for \(a \leq T\).
%Therefore the \(b\)-th Hodge--Tate and de Rham cohomology are abstractly isomorphic,
%for \(b \leq T - 1\).
%The Hodge--Tate decomposition still holds in this range, as \(T - 1 \leq p - 1\).
%Therefore we have an isomorphism between \(H^b_{\mathrm{dR}}(\mathcal{X}/\mathcal{O}_K)_{\mathrm{tor}}\) 
%and the direct sum of their Hodge counterparts,
%implying that the Hodge--de Rham spectral sequence has split degenerate
%torsion up to degrees $T-1$.
By the first paragraph of proof of \cref{main theorem}, we see that the Hodge--de Rham spectral sequence has split degenerate
torsion up to degree $(T-1)$.
The statement about Hodge filtrations being split follows from~\cref{characterizing torsion degenerations}.
The statement about equality of numbers follows from~\cref{referee suggest},
notice that in applying~\cref{referee suggest}, we need the map of next cohomological degree to be injective,
hence lowering the range of application by $1$.
\end{proof}

In this proof, we are not using the duality statement.
In particular, following this proof, if we have $2 (\dim_{\mathcal{X}_0} + 1) \cdot e < p-1$,
then we do not need to invoke the duality statement in the proof of \cref{main theorem}.

In the case when \(e = 1\), namely \(\mathcal{X}_0\) has an unramified lifting,
the result of Fontaine--Messing and Kato gives something slightly more.

\begin{remark}
\label{FM remark}
It is a result of Fontaine--Messing~\cite[Corollary 2.7.(iii)]{FM87} 
and Kato~\cite[Chapter II Proposition 2.5.(1)]{K87}
that given an unramified lift \(\mathcal{X}\),
the Hodge--de Rham spectral sequence degenerates up to degree \(p-1\),
namely all the differentials with \emph{target} of total degree
\(\leq p-1\) are zero.
Moreover they showed (see~\cite[Corollary 2.7.(ii)) and Remark 2.8.(ii)]{FM87} 
and~\cite[Chapter II Proposition 2.5.(2)]{K87}) that the integral Hodge filtrations
on \(H^i_{\mathrm{dR}}(\mathcal{X}/\mathcal{O}_K)\), 
where \(0 \leq i \leq p-1\),
are equipped with divided Frobenius structure and altogether
these form so-called Fontaine--Laffaille modules~\cite{FL83}.
In particular, by a result of Wintenberger~\cite{Win84},
the Hodge filtrations are split submodules in the range \(0 \leq i \leq p-1\).
Hence these results 
%together with the observation that
%Euler characteristic of a flat sheaf is locally constant in a flat family
imply the following:

\begin{corollary}[{Corollary of~\cite[Corollary 2.7]{FM87} or~\cite[Chapter II Proposition 2.5]{K87}}]
Let \(\mathcal{X}\) be an unramified lift of a smooth proper variety
\(\mathcal{X}_0\). 
Then the degree \(\leq p-2\) Hodge numbers of \(X\) 
are determined by \(\mathcal{X}_0\).
%\begin{enumerate}
%\item If \(\dim \mathcal{X}_0 \leq p-2\),
%then the Hodge numbers of \(X\) are determined by \(\mathcal{X}_0\).
%\item In general, the degree \(\leq p-2\) Hodge numbers of \(X\) 
%are determined by \(\mathcal{X}_0\).
%\end{enumerate}
\end{corollary}

Comparing their results concerning unramified liftings with our result,
we see that our approach (specialized to unramified liftings)
so far can only prove analogous facts with 
\(\mathcal{X}_0\) of \(2\) dimension less.
This is due to:
\begin{enumerate}
\item Min's result was established by
proving the \(i\)-th prismatic cohomology of \(\mathcal{X}/\mathfrak{S}\)
shares the same structure as 
the \(i\)-th \'{e}tale cohomology of the (geometric) generic fiber \(X\),
which (in general) only holds when \(i \cdot e < p-1\); and
\item in our general argument we have no control of the contribution
of \(E\)-torsion of prismatic cohomology groups beyond the designated range.
\end{enumerate}
Perhaps both of these two obstacles may be overcome in the unramified case,
which would recover Fontaine--Messing's or Kato's result in this particular direction.
\end{remark}

\subsection{Summary}
\label{summary}
Given \(\mathcal{X}/\mathcal{O}_K\) as in~\cref{notations},
let $T$ be the largest integer satisfying $T \cdot e < p - 1$.
There are following conditions on \(\mathcal{X}/\mathcal{O}_K\):
\begin{enumerate}[label=\textbf{C.\arabic*}]
    \item \label{1} The formal scheme \(\mathcal{X}\) lifts to \(\mathfrak{S}/(E^2)\);
    \item \label{2} The Hodge--Tate spectral sequence has split degenerate torsion
    up to degree $T-1$;
    \item \label{3} The Hodge--Tate spectral sequence has saturated degenerate torsion
    up to degree $T-1$;
    \item \label{4} The Hodge--de Rham spectral sequence has split degenerate torsion
    up to degree $T-1$;
    \item \label{5} The Hodge--de Rham spectral sequence has saturated degenerate torsion
    up to degree $T-1$;
    \item \label{6} The Hodge--de Rham spectral sequence degenerates splittingly;
    \item \label{7} The Hodge--de Rham spectral sequence degenerates saturatedly;
    \item \label{8} The virtual Hodge numbers of \(\mathcal{X}_0\) equal the Hodge numbers of \(X\).
\end{enumerate}
%Note that for conditions~\ref{2}--~\ref{6}, 
%we may single out a range of degree where these conditions apply to.

The relations between these conditions are summarized in the following diagram:
\[
\xymatrix{
\ref{1} \ar@{=>}[r]^{\alpha} & \ref{2} \ar@{=>}[r] \ar@{<=>}[d]^{\beta} 
& \ref{3} \ar@{<=>}[d]^{\beta} & \\
                             & \ref{4} \ar@{=>}[r] \ar@{<=>}[d]^{\delta} 
& \ref{5} \ar@{<=>}[d]^{\delta} & \\
e = 1 \ar@{=>}[uu] \ar@{=>}[ru]^{\gamma}  & \ref{6}   \ar@{=>}[r]   & 
\ref{7}   \ar@{=>}[r]^{\epsilon}   & \ref{8}.
}
\]
Below we remind readers under what condition (and why) 
we have some of these implications:
\begin{itemize}
\item \(\alpha\) holds when \(\dim X \leq p-1\) and follows
from~\cite[Remark 4.13 and Proposition 4.14]{BS19};
\item \(\beta\) follows from Min's work~\cite{Min19}
together with the analysis of relevant spectral sequences,
see~\cref{equivalence of SS};
\item \(\gamma\) holds provided \(\dim X \leq p-2\)
and follows from the work of either 
Fontaine--Messing~\cite[Corollary 2.7]{FM87}
or Kato~\cite[Chapter II Proposition 2.5]{K87};
\item \(\delta\) holds when $\dim\mathcal{X}_0 + 1 \leq T$ 
by duality of de Rham and Hodge cohomologies; and
\item \(\epsilon\) always holds and is the content 
of~\cref{equality of numbers}.
\end{itemize}

% * An Example
\section{Lifting the example of Antieau--Bhatt--Mathew}
\label{lifting ABM}

One might wonder if it is really necessary to have both of conditions (1) and (2)
in~\cref{main theorem}, or even any condition at all,
in order for the Hodge--de Rham spectral sequence to behave nicely.
We would like to mention that in~\cite{Li18}
we found pairs of relatively \(3\)-dimensional
smooth projective schemes over \(\mathbb{Z}_p[\zeta_p]\),
such that their special fibers are isomorphic but the
degree \(2\) Hodge numbers
of their generic fibers are different.
Therefore these give rise to examples where the 
Hodge--de Rham spectral sequence is not saturated degenerate by~\cref{equality of numbers}.
The cohomological degree times ramification index of these examples are
twice of \(p-1\), so this example do not satisfy condition (2) in~\cref{main theorem}.
Moreover, the author suspects that these examples do not satisfy condition (1) 
in~\cref{main theorem} as well.

While it is unclear whether condition (2) in~\cref{main theorem} is really necessary,
in this last section we would like to illustrate, by an example, 
the necessity of condition (1) in~\cref{main theorem}.
More precisely, we shall construct smooth proper schemes over
degree \(2\) ramified extensions of \(\mathbb{Z}_p\),
such that the Hodge--de Rham spectral sequence are not degenerate (starting at degree \(3\)),
and the Hodge filtrations are non-saturated (starting at degree \(2\)).
The idea is to approximate the classifying stack of a lift of \(\alpha_p\)
(which only exists over a ramified ring of integers),
and the key computations and techniques are already in~\cite{ABM19}.

We remark that in a concise paper by W.~Lang~\cite{Lang95}, examples in positive characteristic
which admit liftings to ramified DVRs but violate 
Hodge--de Rham degeneration have been found.
Lang used exactly the idea of approximating $B\alpha_p$, we learned from loc.~cit.~that this idea
dates back to Raynaud~\cite{Ray79}.
In the spirit that our~\Cref{main theorem} 
may be thought of as a generalization of Deligne--Illusie's result,
our example here can also be thought of as a generalization of Lang's.

\subsection{Recollection of~\cite{TO70}}
In this subsection, we give a preliminary discussion of group schemes of order \(p\) over \(p\)-adic base rings.
Fix a scheme \(S\) over \(\mathbb{Z}_p\).
Recall that in~\cite{TO70}, the authors made a detailed study of finite flat group schemes of order \(p\) over \(S\), 
below let us summarize their results.

Firstly, all such group schemes are commutative~\cite[Theorem 1]{TO70}.
Secondly, for each \(p\) there is a unit \(\omega \in \mathbb{Z}_p^*\) (which is denoted as \(\omega_{p-1}\) in loc.~cit.), see \cite[Remark on p.~11]{TO70} for a recursive formula
defining it,
and a bijection between 
\begin{enumerate}
    \item isomorphism classes of finite flat order \(p\) group schemes over \(S\); and
    \item isomorphism classes of triples \((\mathcal{L}, a, b)\) where
    \(\mathcal{L}\) is a line bundle on $S$, elements $a$ and $b$ are sections
    of line bundles:
    \(a \in \Gamma(S, \mathcal{L}^{\otimes (p-1)})\) and
    \(b \in \Gamma(S, \mathcal{L}^{\otimes (1-p)})\), and they satisfy the relation
    \(a \otimes b = p\omega\),
\end{enumerate}
here we have identified 
\(\mathcal{L}^{\otimes (p-1)} \otimes \mathcal{L}^{\otimes (1-p)} \cong \mathcal{O}_S\)~\cite[Theorem 2]{TO70}.
The group associated with \((\mathcal{L}, a, b)\) is denoted by \(\mathrm{G}^{\mathcal{L}}_{a,b}\),
with underlying scheme structure given by 
\(\underline{\mathrm{Spec}}(\mathcal{O}_S \oplus \mathcal{L}^{-1} \oplus \ldots \oplus \mathcal{L}^{-p+1})\)
where the ring structure comes from \(\mathcal{L}^{-p} \xrightarrow{a} \mathcal{L}^{-1}\)~\cite[P.~12]{TO70}.
So \(\mathrm{G}^{\mathcal{L}}_{a,b}\) is an \'{e}tale group scheme if and only if 
\(a \in \Gamma(S, \mathcal{L}^{\otimes (p-1)})\) is an invertible section~\cite[P.~16 Remark 6]{TO70}.
Moreover the Cartier dual of 
\(\mathrm{G}^{\mathcal{L}}_{a,b}\) is\footnote{Note that they are commutative group schemes by first sentence of this paragraph.}
given by \(\mathrm{G}^{\mathcal{L}^{-1}}_{b,a}\)~\cite[P.~15 Remark 2]{TO70}.

\begin{example}
When \(S = \mathrm{Spec}(\mathbb{F}_p)\), there is only one line bundle on \(S\), namely \(\mathcal{O}_S\).
Furthermore we have \(p = 0\) on \(S\). 
Hence we see that group schemes of order \(p\) over \(S\) are classified by pairs \((a,b) \in \mathbb{F}_p^2\)
with the constraint that \(ab = 0\).
Note that these pairs have no nontrivial automorphism as any invertible element \(u \in \mathbb{F}_p^*\) satisfies
\(u^{p-1} = 1\).
There are three possibilities:
\begin{enumerate}
    \item \(a \not= 0\), which forces \(b = 0\), 
    corresponding to a form of the \'{e}tale group scheme \(\mathbb{Z}/p\).
    When \(a = 1\), it is \(\mathbb{Z}/p\).
    \item Dually we can have \(b \not= 0\) and \(a = 0\),
    corresponding to a form of \(\mu_p\). It is \(\mu_p\) when \(b = 1\).
    \item Lastly if both of \(a = b = 0\), we get \(\alpha_p\).
\end{enumerate}

\end{example}

\subsection{A stacky example}
Now we specialize to the case where \(S = \mathrm{Spec}(\mathcal{O}_K)\) is given by
the valuation ring of a \(p\)-adic field.
There is no non-trivial line bundle on a local scheme such as \(S\).
In order to lift \(\alpha_p\) from the residue field of \(\mathcal{O}_K\),
it suffices to find an element \(\pi \in \mathfrak{m}\) such that
\(p/\pi \in \mathfrak{m}\).
Here \(\mathfrak{m}\) denotes the maximal ideal in \(\mathcal{O}_K\).
We see that \(\mathcal{O}_K\) cannot be absolutely unramified,
and as long as it is ramified, we may find such an element \(\pi\).
From now on, let us fix such a choice of \(\mathcal{O}_K\) and \(\pi\).

\begin{notation}
Let \(K\) be a degree \(2\) ramified extension of \(\mathbb{Q}_p\)
with ring of integers \(\mathcal{O}_K\), a uniformizer \(\pi\) in
the maximal ideal \(\mathfrak{m} \subset \mathcal{O}_K\).
Then \(\pi' \coloneqq p \omega/\pi\) is a uniformizer as well.
Denote \(S \coloneqq \mathrm{Spec}(\mathcal{O}_K)\) and let
\(G \coloneqq G^{\mathcal{O}_S}_{\pi, \pi'}\) be the lift of \(\alpha_p\)
over \(S\)
corresponding to \((\pi, \pi')\).
\end{notation}

In the following we shall study the Hodge--Tate and
Hodge--de Rham spectral sequence of \(BG\).
Note that \(BG\) is a smooth proper stack over \(\mathrm{Spec}(\mathcal{O}_K)\)
with special fiber \(BG \times_{\mathcal{O}_K} \mathbb{F}_p \cong B\alpha_p\).
The following computation of Antieau--Bhatt--Mathew is very useful.

\begin{proposition}[{see~\cite[Proposition 4.10]{ABM19}}]
\label{computation of ABM}
If \(p > 2\), the Hodge cohomology group of 
\(B\alpha_p\) is given by
\[
H^*(B\alpha_p, \wedge^*L_{B\alpha_p/\mathbb{F}_p}) \cong
E(\alpha) \otimes P(\beta) \otimes E(s) \otimes P(u)
\]
where \(E(-)\) (resp.~\(P(-)\)) denotes the exterior (resp.~polynomial) algebra
on the designated generator.
Here \(\alpha \in H^1(B\alpha_p, \mathcal{O})\),
\(\beta \in H^2(B\alpha_p, \mathcal{O})\),
\(s \in H^0(B\alpha_p, L_{B\alpha_p/\mathbb{F}_p})\)
and \(u \in H^1(B\alpha_p, L_{B\alpha_p/\mathbb{F}_p})\).
For \(p=2\) we replace \(E(\alpha) \otimes P(\beta)\) with \(P(\alpha)\).
\end{proposition}

\begin{lemma}
The cotangent complex of \(BG\) is 
\(L_{BG/\mathcal{O}_K} \simeq \mathcal{O}/(\pi)[-1]\).
\end{lemma}

\begin{proof}
Observe that the equation of the underlying scheme of \(G\) is given by
\(x^p - \pi x\),
hence we know that \(L_{G/\mathcal{O}_K} \simeq \mathcal{O}_G/(\pi)\).
Therefore the underlying coLie complex of \(G\) is also \(\mathcal{O}_G/(\pi)\).
As \(G\) is commutative,
our statement follows 
from~\cite[Proposition 4.4]{I72}.\footnote{Note that the cotangent complex of \(BG\) is the coLie complex shifted by \(-1\).}
\end{proof}

\begin{remark}
\label{specialize cotangent complex}
In the proof of~\cite[Proposition 4.10]{ABM19}, 
the authors showed that the Postnikov tower
\(\mathcal{O}_{B\alpha_p} \to L_{B\alpha_p/\mathbb{F}_p} \xrightarrow{a} \mathcal{O}_{B\alpha_p}[-1]\)
of the cotangent complex of \(B\alpha_p\)
splits: \(L_{B\alpha_p/\mathbb{F}_p} \simeq 
\mathcal{O}_{B\alpha_p} \oplus \mathcal{O}_{B\alpha_p}[-1]\).
In our situation, we get a triangle in \(D(BG)\):
\[
\mathcal{O}[-1] \to L_{BG/\mathcal{O}_K} \to \mathcal{O}
\]
where the connecting morphism is multiplication by \(\pi\).
Specializing to the special fiber \(B\alpha_p\), we get a diagram
\[
\label{comparing cotangent complex}
\tag{\epsdice{3}}
\xymatrix{
\mathcal{O}[-1] \ar[r] \ar[d] & L_{BG/\mathcal{O}_K} \ar[r] \ar[d] \ar[dl]_{b} \ar[dr]^c & \mathcal{O} \ar[d] \\
\mathcal{O}/(\pi)\cdot u[-1] \ar[r] & L_{B\alpha_p/\mathbb{F}_p} \ar[r] \ar@/^1pc/[l]_a 
& \mathcal{O}/(\pi) \cdot s \\
}
\]
where \(b\) is the identification \(L_{BG/\mathcal{O}_K} \simeq \mathcal{O}/(\pi)[-1]\).
This gives a particular choice of the splitting of
\(L_{B\alpha_p/\mathbb{F}_p} \simeq 
\mathcal{O}_{B\alpha_p} s \oplus \mathcal{O}_{B\alpha_p}u[-1]\),
where the classes \(s\) and \(u\) are as in the statement of the aforementioned~\cref{computation of ABM}.
Here let us name the map \(sp \colon L_{BG/\mathcal{O}_K} \xrightarrow{b \oplus c} 
\mathcal{O}_{B\alpha_p}u[-1] \oplus \mathcal{O}_{B\alpha_p} s
\simeq L_{B\alpha_p/\mathbb{F}_p}\)
for we will use it later.

%Since \(L_{B\alpha_p/\mathbb{F}_p} \simeq L_{BG/\mathcal{O}_K} \otimes^\mathbb{L}_{\mathcal{O}_K} \mathbb{F}_p\), we get a triangle in \(D(BG)\): \(L_{BG/\mathcal{O}_K} \xrightarrow{sp} L_{B\alpha_p/\mathbb{F}_p} \xrightarrow{\delta} L_{BG/\mathcal{O}_K}[1]\). Here \(sp\) denotes the specialization map and \(\delta\) denotes the coboundary map. Observe that \(sp\) identifies \(sp \colon L_{BG/\mathcal{O}_K} = \mathcal{O}/(\pi)[-1] \xrightarrow{\cong} \mathcal{O}_{B\alpha_p}u[-1]\) (K\"{a}hler differential of \(G\) identifies with that of  \(\alpha_p\) under specialization), hence the coboundary map must factor through an identification \(\delta \colon \mathcal{O}_{B\alpha_p} s \xrightarrow{\cong} L_{BG/\mathcal{O}_K}[1]\). Since we will use the specialization map to help understanding the Hodge cohomology groups of \(BG\), from now on let us fix the generator of \(L_{BG/\mathcal{O}_K} = \mathcal{O}/(\pi)u[-1]\). Then we can choose the generator \(s\) such that \(\delta (s) = u\).
\end{remark}

Next let us compute the Hodge cohomology groups of \(BG\) and identify the algebra structure.

\begin{proposition}
\label{hodge cohomology of BG}
For any pair of integers \((i,j)\), we have
\[
H^i(BG, \wedge^j L_{BG/\mathcal{O}_K}) =
\begin{cases} 
      \mathcal{O}_K & i = j =0 \\
      \mathbb{F}_p & j = 0, i = 2m > 0 \text{ or } 0 < j \leq i \\
      0 & \text{otherwise.}
\end{cases}
\]
Therefore specialization maps give rise to injections
\(sp \colon H^i(BG, \wedge^j L_{BG/\mathcal{O}_K}) \hookrightarrow H^i(B\alpha_p, \wedge^jL_{B\alpha_p/\mathbb{F}_p})\)
whenever \(i+j > 0\).
Moreover these injections are compatible with multiplication and differentials,
and gives an identification
\[
H^*(BG, \wedge^* L_{BG/\mathcal{O}_K}) = 
\begin{cases} 
(\mathcal{O}_K[\beta, u] \otimes E(\tau))/(\pi \tau, \pi \beta, \pi u) & p > 2 \\
(\mathcal{O}_K[\beta, u, \tau])/(\pi \tau, \pi \beta, \pi u, \tau^2 - \beta u^2) & p = 2,
\end{cases}
\]
where \(\beta \in H^2(BG, \mathcal{O})\) and \(u \in H^1(BG, L_{BG/\mathcal{O}_K})\)
both specialize to the designated elements in the Hodge ring of \(B\alpha_p\),
and \(\tau \in H^2(BG, L_{BG/\mathcal{O}_K})\) specializes to \(\alpha u + \beta s\)
(up to scale \(s\) by a unit).
\end{proposition}

\begin{proof}
First we begin with the computation of cohomology of \(\mathcal{O}\).
Similar to the first paragraph of proof of~\cite[Proposition 4.10]{ABM19},
we have \(H^*(BG, \mathcal{O}) = \mathrm{Ext}^*_{\mathcal{O}_K[y]/(y^p - \pi' y)} (\mathcal{O}_K, \mathcal{O}_K)\)
by Cartier duality.
Here we used the fact that the Cartier dual of \(G_{\pi, \pi'}\) is \(G_{\pi', \pi}\)
whose underlying scheme structure is \(\mathrm{Spec}(\mathcal{O}_K[y]/(y^p - \pi' y))\)
with its identity section given by \(y = 0\).
Using the standard resolution:
\[
\left( \ldots \mathcal{O}_K[y]/(y^p - \pi' y) \xrightarrow{y^{p-1} - \pi'} \mathcal{O}_K[y]/(y^p - \pi' y) \xrightarrow{y} 
\mathcal{O}_K[y]/(y^p - \pi' y) \right) \simeq \mathcal{O}_K
\]
one verifies the computation when \(j = 0\).

For the case when \(j > 0\), just observe that we have 
\[
\wedge^* L_{BG/\mathcal{O}_K} = \wedge^* \left(\mathcal{O}/(\pi)[-1]\right) = \mathrm{Sym}^* (\mathcal{O}/(\pi))[-*].
\]
Therefore we get
\[
H^i(BG, \wedge^j L_{BG/\mathcal{O}_K}) = H^i(BG, \mathrm{Sym}^j \mathcal{O}/(\pi)[-j])
= H^{i-j}(B\alpha_p, \mathcal{O}),
\]
which verifies the computation when \(j > 0\) via~\cref{computation of ABM}.

The second statement follows from the fact that 
\(H^i(BG, \wedge^j L_{BG/\mathcal{O}_K})\) are all \(\pi\)-torsion when \(i + j > 0\) by the first sentence.
In particular, by dimension consideration we see that the induced map
\(H^2(BG, \mathcal{O}) \to H^2(B\alpha_p, \mathcal{O})\)
must be an isomorphism.
Hence we may pick a generator \(\beta \in H^2(BG, \mathcal{O})\)
which lifts the designated generator in \(H^2(B\alpha_p, \mathcal{O})\).

Next we deal with the statement concerning image of other specialization maps.
Since the map \(b\) in~\ref{comparing cotangent complex} is an identification,
we see that the \(b\) component of
\[
H^*(BG, L_{BG/\mathcal{O}_K}) \xrightarrow{sp = b \oplus c} 
H^*(B\alpha_p, L_{B\alpha_p/\mathbb{F}_p}) = H^{*-1}(B\alpha_p, \mathcal{O}) \cdot u
\oplus H^{*}(B\alpha_p, \mathcal{O}) \cdot s
\]
is always an isomorphism.
In particular we can choose generators of
$H^1(BG, L_{BG/\mathcal{O}_K})$ and $H^2(BG, L_{BG/\mathcal{O}_K})$
corresponding to $u$ and $\alpha \cdot u$ under $b$.
Let us denote each generator by $u$ and $\tau$ respectively.

The map \(c\) factors as the composition of
\(L_{BG/\mathcal{O}_K} \to \mathcal{O} \to \mathcal{O}/(\pi) \cdot s\).
The first map fits in the triangle 
\[
\mathcal{O}[-1] \to L_{BG/\mathcal{O}_K} \cong \mathcal{O}/\pi[-1] 
\to \mathcal{O}
\]
with connecting morphism being multiplication by $\pi$.
Since \(H^1(BG, \mathcal{O})\) is shown to be zero, we see that the map $c$
on $H^1(BG, L_{BG/\mathcal{O}_K})$ factors through zero,
hence $c(u) = 0$.
Still by the computation of \(H^*(BG, \mathcal{O})\), the long exact sequence
of cohomology associated with the above triangle gives an isomorphism
$H^2(BG, L_{BG/\mathcal{O}_K}) \to H^2(BG, \mathcal{O})$.
The map $\mathcal{O} \to \mathcal{O}/(\pi) \cdot s$ is just modulo $\pi$,
hence induces an injection $H^2(BG, \mathcal{O})/\pi \to 
H^2(B\alpha_p, \mathcal{O}/(\pi) \cdot s)$.
Now the cohomology group $H^2(BG, \mathcal{O})$ is shown to be $\mathcal{O}_K/\pi$
and the dimension of $H^2(B\alpha_p, \mathcal{O}/(\pi) \cdot s)$ is $1$,
we see that the map $H^2(BG, \mathcal{O})/\pi \to 
H^2(B\alpha_p, \mathcal{O}/(\pi) \cdot s)$ is an isomorphism.
Therefore we have $c(\tau) = \beta \cdot s$ (up to a unit).
Putting these together, we have
$sp(u) = b(u) + c(u) = u$ and $sp(\tau) = b(\tau) + c(\tau) = \alpha \cdot u
+ \beta \cdot s$.

For the last sentence, let us just prove the case when \(p > 2\),
the case of \(p = 2\) can be proved in the same way.
First we observe that we have 
\[
(\mathcal{O}_K[\beta, u] \otimes E(\tau))
\xrightarrow{f} H^*(BG, \wedge^* L_{BG/\mathcal{O}_K}) \xrightarrow{sp} 
H^*(B\alpha_p, \wedge^*L_{B\alpha_p/\mathbb{F}_p}),
\]
with \(\beta\), \(u\) and \(\tau\) as in the statement.
The map \(f\) must kill the relations
\(\pi \tau, \pi \beta, \pi u\), as the positive degree Hodge groups of \(BG\) are \(\pi\)-torsion.
After quotient out the relations, we get an injection
\[
(\mathcal{O}_K[\beta, u] \otimes E(\tau))/(\pi \tau, \pi \beta, \pi u)
\xrightarrow{sp \circ f} H^*(B\alpha_p, \wedge^*L_{B\alpha_p/\mathbb{F}_p})
\]
on positive degree part because of~\cref{computation of ABM}.
Hence the map \(f\) induces an injection
\[
(\mathcal{O}_K[\beta, u] \otimes E(\tau))/(\pi \tau, \pi \beta, \pi u)
\xrightarrow{f} H^*(BG, \wedge^* L_{BG/\mathcal{O}_K}).
\]
By explicitly comparing dimensions of each bi-degree parts, 
one concludes that \(f\) must also be surjective, hence an isomorphism.
\end{proof}

Finally, we can understand the Hodge--de Rham spectral sequence of \(BG\)
with the aid of~\cite[Proposition 4.12]{ABM19}.

\begin{proposition}
\label{HdR of BG}
In the Hodge--de Rham spectral sequence of \(BG\), we have (up to unit) \(d_1(\tau) = u^2\)
and \(d_1(\beta) = d_1(u) = 0\) for all \(p\).
The de Rham cohomology of \(BG\) is given by
\[
H^*_{\mathrm{dR}}(BG/\mathcal{O}_K) \simeq \mathcal{O}_K[\beta']/(p \beta'),
\]
where \(\beta'\) has degree \(2\).
\end{proposition}

\begin{proof}
The first sentence follows from the proof of~\cite[Proposition 4.12]{ABM19}
and the fact that specialization gives injection
\[
sp \colon H^i(BG, \wedge^j L_{BG/\mathcal{O}_K}) 
\hookrightarrow H^i(B\alpha_p, \wedge^jL_{B\alpha_p/\mathbb{F}_p})
\]
which is compatible with multiplication and differentials.
Indeed, in their proof, the authors show that in the special fiber we have
$d_1(\alpha) = u$ (up to a unit) and all other generators are killed by differentials. 
Hence the specialization of our $d_1(\tau)$ must be $d_1(\alpha u + \beta s) = d_1(\alpha) u = u^2$
(up to a unit).
Since specialization is injective in positive cohomological degrees (by \cref{hodge cohomology of BG}), we conclude that $d_1(\tau) = u^2$.

Using the fact that \(d_1\) is a differential, we see that on the \(E_2\)-page
the non-zero entries are 
\[
E_2^{i,j} =
\begin{cases}
\mathcal{O}_K & i = j = 0 \\
\mathbb{F}_p \cdot \beta^n & i = 0, j = 2n > 0 \\
\mathbb{F}_p \cdot \beta^n u & i = 1, j = 2n + 1 > 0. \\
\end{cases}
\]
In particular, there is no room for nonzero differentials, hence the spectral sequence
degenerates on \(E_2\) page.
In particular, we see that the length of de Rham cohomology is as described in the
statement of this Proposition.
To pin down the \(\mathcal{O}_K\)-module structure of 
\(H^*_{\mathrm{dR}}(BG/\mathcal{O}_K)\), we use the fact that
\(H^i_{\mathrm{dR}}(BG/\mathcal{O}_K)/\pi\) injects into \(H^i_{\mathrm{dR}}(B\alpha_p/\mathbb{F}_p)\)
which is always one-dimensional for \(i \geq 0\) due to~\cite[Proposition 4.10]{ABM19}.

Lastly, pick a preimage of \(\beta\) under 
\(H^2_{\mathrm{dR}}(BG/\mathcal{O}_K) \twoheadrightarrow H^2(BG,\mathcal{O})\), 
denote it by \(\beta' \in 
H^2_{\mathrm{dR}}(BG/\mathcal{O}_K)\).
Since \(H^*_{\mathrm{dR}}(BG/\mathcal{O}_K) \to H^*_{\mathrm{dR}}(B\alpha_p/\mathbb{F}_p)\)
is a map preserving multiplication, we see that \(\beta'^n\) is a generator
of \(H^{2n}_{\mathrm{dR}}(BG/\mathcal{O}_K)\).
This finishes the proof of the ring structure on \(H^*_{\mathrm{dR}}(BG/\mathcal{O}_K)\).
\end{proof}

Similarly, we can understand the Hodge--Tate spectral sequence of \(BG\)
with the aid of~\cite[Remark 4.13]{ABM19}.

\begin{proposition}
\label{HT of BG}
In the Hodge--Tate spectral sequence of \(BG\), we have (up to unit) \(d_2(\tau) = \beta^2\)
and \(d_2(\beta) = d_2(u) = 0\) for all \(p\).
The Hodge--Tate cohomology of \(BG\) is given by
\[
H^*_{\mathrm{HT}}(BG/\mathcal{O}_K) \simeq \mathcal{O}_K[u']/(p u'),
\]
where \(u'\) has degree \(2\).
\end{proposition}

\begin{proof}
Recall that in characteristic $p$, conjugate spectral sequence comes from the canonical filtration on
the de Rham complex of affine opens.
Similarly when working with a prism $(A,I)$,
the Hodge--Tate spectral sequence comes from the canonical filtration on the sheaf $\overline{\Prism}_{-/(A/I)}$.
When the prism is $(\mathbb{Z}_p, (p))$, one can ignore the Frobenius twist, and hence identify the sheaf
$\overline{\Prism}_{-/\mathbb{F}_p}$ with the sheaf of (relative to $\mathbb{F}_p$) de Rham complex
(by either of \cite[Theorem 1.8.(1) or (3)]{BS19}).
Therefore in this case, the Hodge--Tate spectral sequence is identified with the conjugate spectral sequence.

The reduction modulo $u$ gives rise to a map of prisms: $(\mathfrak{S}, (E)) \to (\mathbb{Z}_p, (p))$.
Since $u$ is mapped to $\pi$ under $\mathfrak{S} \to \mathcal{O}_K$, we see that the Hodge--Tate spectral sequence
of $BG$ specializes, under reduction modulo $\pi$, to the conjugate spectral sequence of $B\alpha_p$.

Now the differentials in the conjugate spectral sequence of $B\alpha_p$ are understood in \cite[Remark 4.13]{ABM19}.
Using their Remark, the proof of this Proposition is almost the same as the proof of~\cref{HdR of BG},
except we now have \(d_2(\alpha) = d_2(\beta) = d_2(u) = 0\)
and \(d_2(s) = \beta\) in the special fiber.
The multiplicative structure is justified by the fact that de Rham and Hodge--Tate cohomologies
are the same over $\mathbb{F}_p$. Hence the even degree part of the Hodge--Tate cohomology of $B\alpha_p$
is also a polynomial algebra with a degree $2$ generator. 
\end{proof}

In particular, both of the Hodge--de Rham and Hodge--Tate spectral sequences 
are non-degenerate with nonzero differentials starting at degree \(3\),
and the Hodge (resp.~conjugate) filtrations on de Rham (resp.~Hodge--Tate) 
cohomology is not split starting at degree \(2\).
When the prime is \(p \geq 11\), these give rise to stacky examples satisfying condition (2)
of the main theorem which violates the conclusion.
The obstruction of lifting \(G\) to \(\mathrm{Spec}(\mathfrak{S}/(E^2))\)
specializes (under modulo \(u\)) to the obstruction of lifting 
\(\alpha_p\) to \(\mathrm{Spec}(W_2)\),
which is nonzero. 

Let us take a closer look at degree $2$.
By \cref{HdR of BG} we have a short exact sequence
\[
0 \to H^1(BG, L_{BG/\mathcal{O}_K}) = \mathbb{F}_p \cdot u \to H^2_{\mathrm{dR}}(BG/\mathcal{O}_K) = \left(\mathcal{O}_K/p\right) \cdot \beta'
\]
\[
\to H^2(BG, \mathcal{O}_{BG}) = \mathbb{F}_p \cdot \beta \to 0.
\]
Therefore $\beta'$ lifts $\beta$, and $\pi \cdot \beta' = u$, up to units.
The latter also ``explains'' why $u^2 = 0$ in de Rham cohomology (as $(\pi^2) = (p)$ by our assumption of $e = 2$).
Similarly by \cref{HT of BG} we have a short exact sequence
\[
0 \to H^2(BG, \mathcal{O}_{BG}) = \mathbb{F}_p \cdot \beta \to H^2_{\mathrm{HT}}(BG/\mathcal{O}_K) = \left(\mathcal{O}_K/p\right) \cdot \beta'
\]
\[
\to H^1(BG, L_{BG/\mathcal{O}_K}) = \mathbb{F}_p \cdot u \to 0.
\]
Now up to units, we have $u'$ lifts $u$ and $\beta = \pi \cdot u'$. Again, $\beta^2 = 0$ can be seen
by the fact that $p$ divides $\pi^2$ (actually they only differ by a unit in $\mathcal{O}_K$).

We can also determine the prismatic cohomology of \(BG\) using~\cref{HT of BG}.
Before that, we need a few words about prismatic cohomology of a smooth proper stack.
\begin{remark}
Throughout this remark, let us focus on the Breuil--Kisin prism $(\mathfrak{S},(E))$ associated with $\pi \in \mathcal{O}_K$.

(1) Bhatt--Scholze showed that the prismatic cohomology satisfies
(quasi-)syntomic descent~\cite[Theorem 1.15.(2)]{BS19}.
Hence the presheaf, valued in the derived $\infty$-category $D(\mathfrak{S})$, on
$\mathrm{Syn}^{op}_{\mathcal{O}_K}$ (the syntomic site of $\mathcal{O}_K$)
sending $R$ to $\mathrm{R\Gamma}_{\Prism}(\mathrm{Spf}(\widehat{R})/\mathfrak{S})$ is a sheaf.
Here $\widehat{R}$ denotes the $p$-adic completion of $R$.
Since our group scheme $G$ is syntomic over $\mathcal{O}_K$, we know that $BG$ over $\mathcal{O}_K$ is a syntomic stack,
see \cite[Notation 2.1]{ABM19}.
In \cite[Construction 2.7]{ABM19}, one finds a definition of $\mathrm{R\Gamma}_{\Prism}(BG/\mathfrak{S})$.
Concretely, given any syntomic cover $U \to BG$ with $U$ a syntomic $\mathcal{O}_K$-scheme,
we have
\[
\mathrm{R\Gamma}_{\Prism}(BG/\mathfrak{S}) \simeq \lim_{[m] \in \Delta} \mathrm{R\Gamma}_{\Prism}(\widehat{U}^m/\mathfrak{S}).
\]
Here $\widehat{U}^m$ denotes the $p$-adic formal completion of the $(m+1)$-copies fiber product of $U \to BG$.
The syntomic sheaf property exactly guarantees that this formula does not depend on the choice of the syntomic cover $U \to BG$.

(2) All the results stated previously
concerning prismatic cohomology of smooth proper formal schemes (e.g.~a natural Frobenius structure,~\cref{prisms} and~\cref{structure of prismatic}) 
still hold verbatim for $\mathrm{R\Gamma}_{\Prism}(BG/\mathfrak{S})$.
This is because most of the statements in \cite{BS19} are shown by proving their analogues for affine formal schemes.
Below let us show that all the prismatic cohomology groups of $BG$ are finitely generated.

Since the sheaf \(\Prism_{-/\mathfrak{S}}\) is derived $(p, E)$-complete,
the resulting cohomology groups $H^*_{\Prism}(BG/\mathfrak{S})$ 
are also derived $(p, E)$-complete, as derived completeness is preserved under taking limit.
Since the Hodge--Tate cohomology groups of $BG/\mathfrak{S}$ are finitely generated over $\mathcal{O}_K$
(by \cref{HT of BG}),
derived Nakayama Lemma implies that all the cohomology groups
$H^*_{\Prism}(BG/\mathfrak{S})$ are also finitely generated over $\mathfrak{S}$.

(3) We claim that $H^n_{\Prism}(BG/\mathfrak{S})$ is a Breuil--Kisin module \cite[Definition 4.1]{BMS1} for all $n$.
This follows from the fact that $BG$ is a quasi-compact separated smooth stack over $\mathcal{O}_K$
and the same argument laid out in \cite{BS19}.
Below let us spell out the argument for the sake of rigorous.

For any affine smooth formal scheme $\mathcal{X}$ over $\mathcal{O}_K$, the Frobenius on its prismatic cohomology has an isogeny
property \cite[Theorem 1.15.(4)]{BS19}:
the Frobenius induces a canonical isomorphism
\[
\phi_{\mathfrak{S}}^*\mathrm{R\Gamma}_{\Prism}(\mathcal{X}/\mathfrak{S}) \xrightarrow[\phi_{\mathcal{X}}]{\cong} 
L\eta_{E}\mathrm{R\Gamma}_{\Prism}(\mathcal{X}/\mathfrak{S}).
\]
We direct readers to \cite[Section 6]{BMS1} for a discussion of the $L\eta$ functor.
In particular \cite[Lemma 6.9]{BMS1} implies that, for any $n$, we have a functorial map 
\[
\psi_{\mathcal{X}}^n \colon \tau^{\leq n}\mathrm{R\Gamma}_{\Prism}(\mathcal{X}/\mathfrak{S})
\to \tau^{\leq n} \phi_{\mathfrak{S}}^*\mathrm{R\Gamma}_{\Prism}(\mathcal{X}/\mathfrak{S})
\]
such that its composition with the Frobenius in either direction is given by multiplying $E^n$.

Since the group scheme $G$ is finite flat over $\mathcal{O}_K$, 
we know that $BG$ is quasi-compact, separated, and smooth over $\mathcal{O}_K$.
For a justification of the smoothness, see \cite[\href{https://stacks.math.columbia.edu/tag/0DLS}{Tag 0DLS}]{stacks-project}.
Therefore we may find a smooth cover $X \to BG$ with $X$ being an affine scheme and smooth over $\mathcal{O}_K$.
Together with the separatedness of $BG \to \mathcal{O}_K$, all of the $\widehat{X}^m$'s are smooth affine
formal schemes over $\mathcal{O}_K$.

Using the smooth cover $X \to BG$ in the last paragraph, we have
\[
\mathrm{R\Gamma}_{\Prism}(BG/\mathfrak{S}) \simeq \lim_{[m] \in \Delta} \mathrm{R\Gamma}_{\Prism}(\widehat{X}^m/\mathfrak{S}).
\]
The Frobenius is the totalization of the Frobenius on each of the $\mathrm{R\Gamma}_{\Prism}(\widehat{X}^m/\mathfrak{S})$.
Fix a positive integer $n$, using the relation between canonical truncation and limit, we get:
\[
\tau^{\leq n}\mathrm{R\Gamma}_{\Prism}(BG/\mathfrak{S}) \simeq 
\tau^{\leq n}\lim_{[m] \in \Delta} \tau^{\leq n}\mathrm{R\Gamma}_{\Prism}(\widehat{X}^m/\mathfrak{S}).
\]
Because $\phi_{\mathfrak{S}}$ is finite flat, we also have
\[
\tau^{\leq n}\phi_{\mathfrak{S}}^*\mathrm{R\Gamma}_{\Prism}(BG/\mathfrak{S}) \simeq \tau^{\leq n}\lim_{[m] \in \Delta}
\tau^{\leq n}\phi_{\mathfrak{S}}^*\mathrm{R\Gamma}_{\Prism}(\widehat{X}^m/\mathfrak{S}).
\]
Lastly we totalize the maps $\psi_{\widehat{X}^m}^n$ to get a map
\[
\psi_{BG}^n \colon \tau^{\leq n}\mathrm{R\Gamma}_{\Prism}(BG/\mathfrak{S})
\to \tau^{\leq n} \phi_{\mathfrak{S}}^*\mathrm{R\Gamma}_{\Prism}(BG/\mathfrak{S}).
\]
This map composes with the Frobenius in either direction is given by multiplying $E^n$,
as it is so for all of the maps $\psi_{\widehat{X}^m}^n$.
In particular, we see that $H^n_{\Prism}(BG/\mathfrak{S})$ is a Breuil--Kisin module (of height $n$).
\end{remark}

Finally we are ready to compute the Breuil--Kisin prismatic cohomology of $BG$.

\begin{proposition}
\label{prismatic of BG}
The prismatic cohomology of \(BG\) is given by
\[
H^*_{\Prism}(BG/\mathfrak{S}) \simeq \mathfrak{S}[\tilde{u}]/(p \tilde{u}),
\]
where \(\tilde{u}\) has degree \(2\).
%and is \(\phi\)-invariant.
\end{proposition}

Before giving the proof, we need an auxiliary lemma.

\begin{lemma}
Let \(M\) be a cyclic torsion Breuil--Kisin module over \(\mathfrak{S}\) with no \(E\)-torsion,
then there is an integer \(n\) such that \(M \simeq \mathfrak{S}/(p^n)\).
\end{lemma}

\begin{proof}
Say \(M = \mathfrak{S}/I\).
First we know that \(M[1/p] = 0\), see~\cite[Proposition 4.3]{BMS1}.
Since \(M\) has no \(E\)-torsion, we see that \(M\) contains no nonzero finite \(\mathfrak{S}\)-submodule.
Let \(n\) be the smallest integer such that \(p^n \in I\).
It suffices to show that for any non-unit \(f \in \mathfrak{S}-(p)\), the smallest integer \(m\) such that
\(p^m f \in I\) is \(n\).
Suppose otherwise, then we have \(m < n\).
Then the image of
\[
\mathfrak{S}/(f, p) \xrightarrow{\cdot p^{n-1}} M 
\]
is a nonzero (as \(p^{n-1} \not\in I\)) finite (as the image of \(f\) in \(\mathfrak{S}/(p) = k[\![u]\!]\)
is nonzero and non-unit) submodule, which we have argued is impossible.
Therefore we must have \(n = m\).
\end{proof}

\begin{proof}[Proof of~\cref{prismatic of BG}]
The Hodge--Tate specialization gives us short exact sequences:
\[
0 \to H^*_{\Prism}(BG/\mathfrak{S})/(E) \to H^*_{\mathrm{HT}}(BG/\mathcal{O}_K) \to
H^{*+1}_{\Prism}(BG/\mathfrak{S})[E] \to 0,
\]
where \(M[E]\) denotes the \(E\)-torsion of an \(\mathfrak{S}\)-module \(M\).

We make the following claim:
\begin{enumerate}
    \item the odd degree prismatic cohomology groups of \(BG\) are zero; and
    \item the positive even degree prismatic cohomology groups of \(BG\) are cyclic
    and \(E\)-torsion free.
\end{enumerate}
Indeed by our~\cref{HT of BG}, we see that $H^{odd}_{\Prism}(BG/\mathfrak{S})/(E) = 0$.
Since $H^i_{\Prism}(BG/\mathfrak{S})$ is $E$-complete, this gives (1) above.
Using the above short exact sequence and vanishing of odd degree Hodge--Tate cohomology
established in \cref{HT of BG}, we get that the positive even degree
prismatic cohomology groups of \(BG\) are \(E\)-torsion free.
Then we use \cref{HT of BG} again, to see that for any $i > 0$ we get an isomorphism
$H^{2i}_{\Prism}(BG/\mathfrak{S})/(E) \cong \mathfrak{S}/p \cdot \tilde{u}^i$.
Therefore for each $i$ we can find a map $\mathfrak{S} \to H^{2i}_{\Prism}(BG/\mathfrak{S})$,
which is surjective after modulo $E$.
Since $\mathfrak{S}$ itself is derived $(p,E)$-complete, the cokernel of this map
is both $E$-adically complete and vanishes after modulo $E$.
These two together imply that the cokernel vanishes, in other words, the chosen map
$\mathfrak{S} \to H^{2i}_{\Prism}(BG/\mathfrak{S})$ is surjective,
which shows (2) above.

For any $m$ we have that \(H^{2m}_{\Prism}(BG/\mathfrak{S})\) is cyclic and \(E\)-torsion free,
hence it is either free or isomorphic to \(\mathfrak{S}/(p^n)\) for some \(n\)
(by the above Lemma).
To see that we must be in the latter case with \(n\) being \(1\), 
we use the fact that it is so under the Hodge--Tate specialization.
Powers of any generator in \(H^{2}_{\Prism}(BG/\mathfrak{S})\) 
are generators of the corresponding
prismatic cohomology group, as it is so after the Hodge--Tate specialization
(using again $E$-adically completeness of these prismatic cohomology groups).

%Lastly, we observe that the Frobenius on \(H^2(B\alpha_p, \mathcal{O}\) is identity, which implies that after modulo \(u\) there is a Frobenius invariant generator, therefore we can find a Frobenius invariant generator in \(H^2_{\Prism}(BG/\mathfrak{S})\).
\end{proof}

\begin{remark}
We do not know how Frobenius acts on the prismatic cohomology groups.
Since the geometric generic fiber of \(BG\) is \(B \mathbb{Z}/p\), we
at least know that the Frobenius is not identically zero, by \'{e}tale specialization
of the prismatic cohomology~\cite[Theorem 1.8.(4)]{BS19}.
On the other hand, since Frobenius is zero for \(B\alpha_p\),
we know that Frobenius also cannot be surjective on \(H^2_{\Prism}(BG/\mathfrak{S})\).
Since \(((\mathbb{F}_p[\![u]\!])^*)^{p-1} = (1 + (u), \times)\),
we see that after choosing an appropriate generator, the Frobenius on
\(H^2_{\Prism}(BG/\mathfrak{S}) \simeq \mathfrak{S}/(p) \cong \mathbb{F}_p[\![u]\!]\)
sends \(1\) to \( \gamma \cdot u^d\), where \(\gamma \in \mathbb{F}_p^*\) and 
\(d\) is a positive integer.
It would be interesting to understand the relation between our 
choice\footnote{Recall that we need to make such a choice in order to lift \(\alpha_p\).} 
of \(\pi\) and the values \(\gamma\) and \(d\).
\end{remark}

\subsection{Approximating \(BG\)}
In this last subsection, let us show that the pathologies of \(BG\)
are inherited by approximations of \(BG\), so that in the end we can get some scheme examples.
For this purpose, it suffices to follow~\cite[Section 6]{ABM19}.

\begin{proposition}[{See also~\cite[Theorem 1.2]{ABM19}}]
For any integer \(d \geq 0\), there exists a smooth projective
\(\mathcal{O}_K\)-scheme \(\mathcal{X}\) of dimension \(d\) together with a map 
\(\mathcal{X} \to BG\) such that the pullback 
\(H^i(BG, \wedge^j L_{BG/\mathcal{O}_K}) \to 
H^i(\mathcal{X}, \wedge^j L_{\mathcal{X}/\mathcal{O}_K})\)
is injective for \(i + j \leq d\).
\end{proposition}

\begin{proof}
We simply follow the first paragraph of~\cite[Section 6, proof of Theorem 1.2]{ABM19}.
By standard argument (see e.g.~\cite[2.7-2.9]{BMS1}), 
we can find an integral representation \(V\) of \(G\) 
and a \(d\)-dimensional complete intersection \(\mathcal{Y} \subset \mathbb{P}(V)\)
such that \(\mathcal{Y}\) is stable under the \(G\)-action, the action is free,
and \(\mathcal{X} \coloneqq \mathcal{Y}/G \simeq [\mathcal{Y}/G]\) 
is smooth and projective over \(\mathcal{O}_K\) together with a map \(\mathcal{X} \to BG\).
We see that the special fiber of this map induces injections on the corresponding
Hodge cohomology groups.
Now we observe that the composite map
\[
H^i(BG, \wedge^j L_{BG/\mathcal{O}_K}) \to H^i(B\alpha_p, \wedge^j L_{B\alpha_p/\mathbb{F}_p})
\to H^i(\mathcal{X}_0, \Omega^j_{\mathcal{X}_0/\mathbb{F}_p})
\]
is injective when \(i+j \leq d\) (as it is composite of two injective maps) and factors through
\(H^i(BG, \wedge^j L_{BG/\mathcal{O}_K}) \to 
H^i(\mathcal{X}, \Omega^j_{\mathcal{X}/\mathcal{O}_K})\).
Hence the latter map must also be injective when \(i + j \leq d\).
\end{proof}

By choosing \(d = 4\) and using~\cref{HdR of BG} and~\cref{HT of BG}, 
we arrive at the following theorem.

\begin{theorem}
\label{main example}
There exists a smooth projective relative \(4\)-fold \(\mathcal{X}\)
over a ramified degree two extension \(\mathcal{O}_K\) of \(\mathbb{Z}_p\), such that
both of its Hodge--de Rham and Hodge--Tate spectral sequences are non-degenerate.
Moreover the Hodge/conjugate filtrations are non-split as \(\mathcal{O}_K\)-modules.
\end{theorem}

% * End of document
% ** Bibliography
%\bibliographystyle{amsalpha}
%\bibliography{int_hodge_fil}

\end{document}